\documentclass{paper}       

\usepackage{graphicx,subfigure,paralist}
\usepackage{latexsym,amsfonts,amssymb,amsmath,enumerate,framed,epsfig,amsthm}
\usepackage{color}
\usepackage{authblk}

\newcommand{\sgn} {\operatorname{sign}}

\newcommand{\Por}{{\Pi}^r}
\newcommand{\Po}{\Pi}

\newcommand{\Bildeinbinden}[2] {\setlength{\epsfxsize}{#1}\epsfbox{#2}}

\newtheorem{theorem}{Theorem}[section]
\newtheorem{corollary}{Corollary}[section]
\newtheorem{lemma}{Lemma}[section]

 \numberwithin{equation}{section}

\allowdisplaybreaks
\begin{document}

\title{On Sharpness of Error Bounds for Single Hidden Layer Feedforward Neural
Networks}

\author{Steffen Goebbels\footnote{Steffen.Goebbels@hsnr.de, Niederrhein
University of Applied Sciences, Faculty of Electrical Engineering and Computer Science, Institute for Pattern
Recognition, D-47805 Krefeld, Germany}}

\maketitle
\begin{center}
Sixth uploaded version: This is a pre-print of an accepted journal paper
that will appear in Results in Mathematics (RIMA) under the title ``{\bf On Sharpness
of Error Bounds for Univariate Approximation by Single Hidden Layer Feedforward
Neural Networks}'', 17.06.2020
\end{center}

\begin{abstract}
A new non-linear variant of a quantitative extension of the uniform boundedness
principle is used to show sharpness of error bounds for univariate approximation
by sums of sigmoid and ReLU functions. Single hidden layer feedforward neural
networks with one input node perform such operations. Errors of best
approximation can be expressed using moduli of smoothness of the function to be approximated (i.e., to
be learned). In this context, the quantitative
extension of the uniform boundedness principle indeed allows to construct
counterexamples that show approximation rates to be best
possible. Approximation errors do not belong to the little-o class of given
bounds. By choosing piecewise linear activation functions, the discussed
problem becomes free knot spline approximation.
Results of the present paper also hold for non-polynomial
(and not piecewise defined) activation functions like inverse tangent. 
Based on Vapnik-Chervonenkis dimension, first results are shown for
the logistic function.
\end{abstract}

\\[1em]

\section{Introduction} 
A feedforward neural network with an activation function $\sigma$, one input,
one output node, and one hidden layer of $n$ neurons as shown in
Figure \ref{fignet} implements an univariate real function $g$ of type 
$$g(x)=\sum_{k=1}^n a_k \sigma(b_k x+c_k) $$
with weights $a_k, b_k, c_k\in\mathbb{R}$.
\begin{figure}
\begin{center}
\Bildeinbinden{0.5\columnwidth}{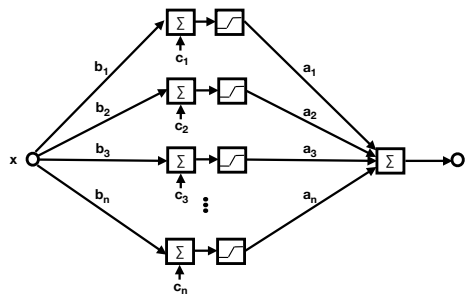}
\end{center}
\caption{One hidden layer neural network realizing $\sum_{k=1}^n a_k
\sigma(b_k x+c_k)$\label{fignet}}
\end{figure}
The given paper does not deal with multivariate approximation.
But some results can be extended to multiple
input nodes, see Section \ref{secFuture}.
Often, activation functions are sigmoid. A sigmoid function $\sigma :
\mathbb{R}\to\mathbb{R}$ is a measurable function with 
$$\lim_{x\to -\infty} \sigma(x)=0 \text{ and }
\lim_{x\to \infty} \sigma(x)=1.$$ 
Sometimes also monotonicity, boundedness, continuity, or even differentiability
may be prescribed. Deviant definitions are based on convexity and
concavity.
In case of differentiability, functions have a bell-shaped first derivative.
Throughout this paper, approximation properties of following sigmoid functions
are discussed:
\begin{alignat*}{3}
\sigma_h(x) &:= \left\{ \begin{array}{cc}0,& x<0\\ 1,& x\geq
0\end{array}\right.&& \text{ (Heaviside function)},\\ 
\sigma_c(x) &:= \left\{ \begin{array}{cl}0,& x<-\frac{1}{2}\\
x+\frac{1}{2} ,&-\frac{1}{2} \leq x\leq \frac{1}{2}\\ 1,&
x>\frac{1}{2}\end{array}\right.&&\text{ (cut function)},\\
\sigma_a(x) &:=
\frac{1}{2}+\frac{1}{\pi} \arctan(x)&&\text{ (inverse tangent)},\\
\sigma_l(x) &:=
\frac{1}{1+e^{-x}}=\frac{1}{2}\left(1+\tanh\left(\frac{x}{2}\right)\right)&&\text{
(logistic function)}.
\end{alignat*}
Although not a sigmoid function, the ReLU function (Rectified Linear Unit)
$$\sigma_r(x):=\max\{0,x\}$$ 
is often used as activation function for deep neural networks due to its
computational simplicity. The
Exponential Linear Unit (ELU) activation function 
$$ \sigma_e(x) := \left\{ \begin{array}{cl} \alpha(e^x-1),& x<0\\
x ,& x\geq 0\end{array}\right. $$
with parameter $\alpha\neq 0$ is a smoother variant of ReLU for $\alpha=1$.
%


Qualitative approximation properties of neural networks have been
studied extensively. For example, it is possible to choose an infinitely often
differentiable, almost monotonous, sigmoid
activation function $\sigma$ such that for each continuous function $f$,
each compact interval and each bound $\varepsilon>0$ weights $a_0, a_1, b_1,
c_1\in\mathbb{R}$ exist such that $f$ can be 
approximated uniformly by $a_0+ a_1 \sigma(b_1 x+c_1)$
on the interval within bound $\varepsilon$, see
\cite{GULIYEV2018296} and literature cited there.
In this sense, a neural network with only one hidden neuron is capable of approximating every
continuous function. 
However, activation
functions typically are chosen fixed when applying neural networks to solve
application problems. They do not depend on the  unknown function to be
approximated.
In the late 1980s it was already known
that, by increasing the number of neurons, 
all continuous functions can be
approximated arbitrarily well in the sup-norm on a compact set with each
non-constant bounded monotonically increasing and continuous activation
function (universal approximation or density property, see proof of Funahashi in \cite{Funahashi89}).
For each continuous sigmoid
activation function (that does not have to be monotone), the universal
approximation property was proved by Cybenko in \cite{Cybenko89} on the unit
cube.
The
result was extended to bounded sigmoid activation functions by
Jones \cite{Jones90} without requiring continuity or monotonicity. For
monotone sigmoid (and not necessarily continuous) activation functions, Hornik,
Stinchcombe and White \cite{HORNIK1989359} extended the universal approximation property 
to the approximation of measurable functions.
Hornik \cite{Hornik91} proved density in $L^p$-spaces for any 
non-constant bounded and continuous activation functions. A rather general
theorem is proved in \cite{Leshno93}. Leshno et al.~showed for
any continuous activation function $\sigma$
that the universal approximation property is equivalent
to the fact that $\sigma$ is not an algebraic polynomial.

There are many other density results, e.g., cf.\ \cite{Chen92,Costarelli2013}.
Literature overviews can be found in \cite{pinkus_1999} and
\cite{Sanguineti08universalapproximation}.

To approximate or interpolate a given but unknown function $f$,
constants $a_k$, $b_k$, and $c_k$ typically are obtained by learning
based on sampled function values of $f$. The underlying optimization
algorithm (like gradient descent with back propagation) might get stuck in a
local but not in a global minimum.
Thus, it might not find optimal constants to approximate $f$ best possible.
This paper does not focus on learning but on general approximation
properties of function spaces $$ \Phi_{n,\sigma}:=\left\{g : [0,1]\to\mathbb{R} : g(x)=\sum_{k=1}^n a_k
\sigma(b_k x+c_k) :
a_k,b_k,c_k\in\mathbb{R}\right\}.$$ 
Thus, we discuss functions on the interval $[0, 1]$. Without loss of
generality, it is used instead of an arbitrary compact interval $[a, b]$.
In some papers, an additional constant function $a_0$ is allowed as summand in
the definition of $\Phi_{n,\sigma}$. Please note that $a_k \sigma(0\cdot
x + b_k)$ already is a constant and that the definitions do not differ
significantly.

For $f\in B[0, 1]$, the set of bounded real-valued functions on the interval
$[0, 1]$, let $\|f\|_{B[0, 1]} := \sup\{
|f(x)| : x\in[0, 1]\}$. By $C[0, 1]$ we denote the Banach space of continuous
functions on $[0, 1]$ equipped with norm $\|f\|_{C[0, 1]}:= \|f\|_{B[0, 1]}$.
For Banach spaces $L^p[0, 1]$, $1\leq p<\infty$, of measurable functions we
denote the norm by $\|f\|_{L^p[0, 1]}:=(\int_0^1 |f(x)|^p\, dx)^{1/p}$. To avoid
case differentiation, we set
$X^\infty[0, 1]:=C[0, 1]$ with $\|\cdot\|_{X^\infty[0,
1]}:=\|\cdot\|_{B[0, 1]}$, and $X^p[0, 1]:=L^p[0, 1]$ with $\|f\|_{X^p[0,
1]}:=\|f\|_{L^p[0, 1]}$, $1\leq p<\infty$.

The error of best approximation
$E(\Phi_{n,\sigma}, f)_p$, $1\leq p\leq \infty$, is defined via
$$ E(\Phi_{n,\sigma}, f)_p := \inf\{ \|f-g\|_{X^p[0, 1]} : g\in \Phi_{n,\sigma}
\}.
$$
We use the abbreviation $E(\Phi_{n,\sigma}, f):=E(\Phi_{n,\sigma},
f)_\infty$ for $p=\infty$.

A trained network cannot approximate a function better than the error of best
approximation. Therefore, it is an important measure of what can and what cannot
be done with such a network.

The error of best approximation depends on the
smoothness of $f$ that is measured in terms of moduli of smoothness (or moduli
of continuity). 
In contrast to using derivatives,
first and higher differences of $f$ obviously always exist. By applying a norm
to such differences, moduli of smoothness measure a ``degree of continuity'' of $f$.

For a natural number $r \in \mathbb{N}=\{1, 2,\dots\}$, the $r$th difference of
$f\in B[0, 1]$ at point $x\in [0, 1-rh]$ with step size $h>0$ is defined as 
\[ \Delta^1_h
f(x):= f(x+h)-f(x),\quad \Delta^r_h f(x):= \Delta^1_h \Delta^{r-1}_h f(x),\, r>1,
\mbox{ or}
\]
\[
 \Delta^r_h f(x) := \sum_{j=0}^r (-1)^{r-j}\binom{r}{j} f(x+jh).
\]
The $r$th uniform modulus of smoothness is the
smallest upper bound of the absolute values of $r$th differences: 
\[  \omega_r(f, \delta) := \omega_r
(f, \delta)_\infty := \sup \left\{ |\Delta_h^r f(x)| : x\in [0,1-rh],\,
0<h\leq\delta\right\}.
\]
With respect to $L^p$ spaces, $1\leq p<\infty$, let
\[ \omega_r
(f, \delta)_p := \sup \left\{ \left(\int_0^{1-rh} |\Delta_h^r f(x)|^p \,
dx\right)^{1/p}\, :\, 0<h\leq\delta\right\}.
\]
Obviously, $ \omega_r
(f, \delta)_p \leq 2^r \|f\|_{X^p[0,1]}$, and for $r$-times
continuously differentiable functions $f$, there holds
(cf.~\cite[p.~46]{DeVore93})
\begin{equation}
 \omega_r(f, \delta)_p \leq \delta^r\|f^{(r)}\|_{X^p[0, 1]}.\label{gegenAbl}
\end{equation}

Barron applied Fourier methods in \cite{Barron93}, cf.~\cite{LEWICKI2004}, 
to establish rates of convergence in an $L^2$-norm, i.e., he estimated the
error $E(\Phi_{n,\sigma}, f)_2$ with respect to $n$ for $n\to\infty$.
Makovoz \cite{MAKOVOZ1998215} analyzed rates for uniform convergence. With
respect to moduli of smoothness, Debao \cite{Chen93} proved a direct estimate
that is here presented in a version of the textbook \cite[p.~172ff]{Cheney2000}.
This estimate is independent of the choice of a
bounded, sigmoid function $\sigma$.
Doctoral thesis \cite{Costarelli14}, cf.\ \cite{COSTARELLI2013101}, provides an
overview of such direct estimates in Section 1.3. 
%

According to Debao,
\begin{equation}
 E(\Phi_{n,\sigma}, f) \leq \left[\sup_{x\in\mathbb{R}} |\sigma(x)|\right]
 \omega_1\left(f, \frac{1}{n}\right)\label{d1}
\end{equation}
holds for each $f\in C[0,
1]$.
This is the prototype estimate for which sharpness is discussed in this paper.
In fact, the result of Debao for $E(\Phi_{n,\sigma}, f)$ allows to
additionally restrict weights such that $b_k\in\mathbb{N}$ and
$c_k\in\mathbb{Z}$.
The estimate has to hold true even for $\sigma$ being a discontinuous
Heaviside function. That is the reason why one can only expect an estimate
in terms of a first order modulus of smoothness. 
If the order of approximation
of a continuous function $f$ by such piecewise constant functions is $o(1/n)$
then $f$ itself is a constant, see \cite[p.~366]{DeVore93}.
In fact, the idea behind Debao's proof is that sigmoid functions can be
asymptotically seen as Heaviside functions. One gets arbitrary step functions to
approximate $f$ by superposition of Heaviside
functions.
For quasi-interpolation operators based on the logistic activation function
$\sigma_l$, Chen and Zhao proved similar estimates in \cite{Chen13}
(cf.~\cite{ANASTASSIOU2011809,Anastassiou2011} for hyperbolic tangent).
However, they only reach a convergence order of $O\left(1/n^\alpha\right)$ for
$\alpha<1$. With respect to the error of best approximation, they prove 
$$E(\Phi_{n,\sigma_l}, f) \leq 80\cdot
 \omega_1\left(f, \frac{\exp\left(\frac{3}{2}\right)}{n}\right)$$ 
by estimating
with a polynomial of best approximation. Due to the different technique,
constants are larger than in error bound (\ref{d1}). 

If one takes additional properties
of $\sigma$ into account, higher convergence rates are possible.
Continuous sigmoid cut function $\sigma_c$ and ReLU function $\sigma_r$ 
lead to spaces $\Phi_{n,\sigma_{c,r}}$ of continuous,
piecewise linear functions.
They consist of free knot spline functions  
of polynomial degree at most one
with at most $2n$ or $n$ knots, cf.\ \cite[Section 12.8]{DeVore93}.
Spaces $\Phi_{n,\sigma_{c,r}}$ include all continuous spline
functions $g$ on $[0,1]$ with polynomial degree at most one that have at most
$n-1$ simple knots.
We show $g\in \Phi_{n,\sigma_c}$ for such a spline $g$ with equidistant knots
$x_k=\frac{k}{n-1}$, $0\leq k \leq n-1$, to obtain an error bound for $n\geq 2$:
\begin{alignat*}{1}
g(x) &= g(0) \sigma_c\left(x\!+\!\frac{1}{2}\right)+
\sum_{k=0}^{n-2}
\left[g\left( x_{k+1}\right)-g\left( x_{k}\right)\right] 
\sigma_c\left( (n\!-\!1) \cdot
\left(x- \frac{k+\frac{1}{2}}{n-1}\right)\right).
\end{alignat*}
One can also represent $g$ by ReLU functions, i.e.,
$g\in
\Phi_{n,\sigma_r}$:
With ($1\leq k\leq n-1$)
$$  m_0 := g(0),\quad m_k := 
  \frac{
  g\left(x_{k}\right)-g\left(x_{k-1}\right)}{x_k-x_{k-1}}
  -\sum_{j=0}^{k-1} m_j,
  $$  
we get
\begin{equation*}
 g(x) = m_0\cdot \sigma_r\left(x+1\right) + \sum_{k=1}^{n-1} m_k
\sigma_r\left(x-x_{k-1}\right).
\end{equation*}
Thus, we can apply the larger error bound \cite[p.~225, Theorem 7.3 for
 $\delta=(n-1)^{-1}$]{DeVore93} for fixed simple knot spline approximation 
with functions $g$ in terms of a second modulus to
improve convergence rates up to $O\left(1/n^2\right)$: For $f\in X^p[0, 1]$, $1\leq p\leq\infty$,
and $n\geq 2$
\begin{equation}
E(\Phi_{n,\sigma_{c,r}}, f)_p \leq C \omega_2\left(f,
\frac{1}{n-1}\right)_p \leq C \omega_2\left(f,
\frac{2}{n}\right)_p \leq 4C \omega_2\left(f,
\frac{1}{n}\right)_p.\label{m2}
\end{equation}

Section \ref{secHigh} deals with even higher order direct estimates. 
Similarly to (\ref{m2}), not only sup-norm bound (\ref{d1}) but also an
$L^p$-bound, $1\leq p< \infty$, for approximation with Heaviside function
$\sigma_h$ can be obtained from the corresponding larger bound of 
fixed simple knot spline approximation.
Each $L^p[0, 1]$-function that is constant between knots $x_k=\frac{k}{n}$,
$0\leq k\leq n$, can be written as a linear combination of $n$ translated
Heaviside functions.
Thus, \cite[p.~225, Theorem
7.3 for $\delta=1/n$]{DeVore93}) yields for $n\in\mathbb{N}$
\begin{equation}
E(\Phi_{n,\sigma_h}, f)_p \leq C \omega_1\left(f,
\frac{1}{n}\right)_p.\label{m2b}
\end{equation}

Lower error bounds are much harder to obtain than upper bounds, 
cf.\ \cite{pinkus_1999} for some results with regard to multilayer feedforward
perceptron networks. Often, lower bounds are given using 
a (non-linear) Kolmogorov $n$-width $W_n$ (cf.\ \cite{Pinkus85,Sanguineti99}),
$$W_n := \inf_{b_1,\dots,b_n, c_1\dots,c_n} \sup_{f\in X}
\inf_{a_1,\dots,a_n}\left\| f(\cdot)- \sum_{k=1}^n a_k\sigma(b_k\, \cdot\, +
c_k)\right\|$$ 
for a suitable function space $X$ (of functions with certain smoothness)
and norm $\|\cdot\|$. Thus, parameters
$b_k$ and $c_k$ cannot be chosen individually for each function $f\in X$.
Higher rates of convergence might occur, if that becomes possible.

There are three somewhat different types of sharpness results that might be able
to show that left sides of equations (\ref{d1}), (\ref{m2}),  (\ref{m2b}) or
(\ref{jacksonneu}) and (\ref{jacksonneu2}) in Section \ref{secHigh} do not
converge faster to zero than the right sides.

The most far reaching results would provide lower estimates of errors of best
approximation in which the lower bound is a modulus of smoothness. In
connection with direct upper bounds in terms of the same moduli,
this would establish theorems similar to the equivalence
between moduli of smoothness and K-functionals (cf.\ \cite[theorem
of Johnen, p.~177]{DeVore93}, \cite{JohnenScherer}) in
which the error of best approximation replaces the K-functional. Let $\sigma$ be $r$-times continuously differentiable
like $\sigma_a$ or $\sigma_l$. Then for $f\in C[0, 1]$, a
standard estimate based on (\ref{gegenAbl}) is
\begin{alignat*}{1}
 \omega_r\left(f,\frac{1}{n}\right) &= \inf_{g\in \Phi_{n,\sigma}}\!
 \omega_r\left(f-g+g,\frac{1}{n}\right)  \leq \inf_{g\in \Phi_{n,\sigma}}\! 
\left[\omega_r\left(f-g,\frac{1}{n}\right) +
\omega_r\left(g,\frac{1}{n}\right)\right]\\ 
&\leq  \inf_{g\in \Phi_{n,\sigma}}\! 
\left[ 2^r\|f-g\|_{B[0, 1]} + \frac{1}{n^r}\|g^{(r)}\|_{B[0,
1]}\right].\end{alignat*}  
It is unlikely that one can somehow bound $\|g^{(r)}\|_{B[0, 1]}$ by
$C\|f\|_{B[0, 1]}$ to get
\begin{equation}
 \omega_r\left(f,\frac{1}{n}\right) \leq 2^r E(\phi_{n,\sigma},f) +
\frac{C}{n^r}\|f\|_{B[0, 1]}. \label{lb}
\end{equation}
However, there are different attempts to prove such
theorems in \cite[Remark 1, p.~620]{Wang10}, \cite[p.~101]{Xu2004} and
\cite[p.~451]{Xu2006}. But the proofs contain
difficulties and results are wrong, see \cite{Goebbels19}. In fact, we prove
with Lemma \ref{lower} in Section \ref{secHigh} that (\ref{lb}) is not valid even if constant
$C$ is allowed to depend on $f$.


A second class of sharpness results consists of inverse and equivalence
theorems.
Inverse theorems provide upper bounds for moduli of smoothness in terms of
weighted sums of approximation errors. For pseudo-interpolation operators based on piecewise linear activation
functions and B-splines (but not for errors of best approximation), 
\cite{Li2019} deals with an inverse estimate based on
Bernstein polynomials.

An idea that does not work is to adapt the inverse
theorem for best trigonometric approximation \cite[p.~208]{DeVore93}. 
Without considering effects related to interval endpoints in algebraic
approximation one gets a (wrong) candidate inequality
\begin{equation}
 \omega_r\left(f,\frac{1}{n}\right) \leq \frac{C_r}{n^r} \left[\sum_{k=1}^n
k^{r-1} E(\phi_{k,\sigma},f) + \|f\|_{B[0,1]}\right].\label{proto} 
\end{equation}
By choosing $f\equiv\sigma$ for a non-polynomial, $r$th times continuously
differentiable activation function $\sigma$, the modulus on the left side of the estimate
behaves like $n^{-r}$. But the errors of best
approximation on the right side are zero. At least this can be cured by the
additional expression $\frac{C_r}{n^r} \|f\|_{B[0,1]}$.

Typically, the proof of an inverse theorem is based on a Bernstein-type
inequality that is difficult to formulate for function spaces discussed here. The Bernstein
inequality provides a bound for derivatives. If $p_n$ is a trigonometric
polynomial of degree at most $n$ then $\|p_n'\|_{B[0,2\pi]}\leq n \|p_n\|_{B[0,2\pi]}$,
cf.~\cite[p.~97]{DeVore93}. The problem here is that
differentiating $a\sigma(bx+c)$ leads to a factor $b$ that cannot be
bounded easily. Indeed, we show for a large class of activation functions that
(\ref{proto}) can't hold, see (\ref{antiproto}). As noticed in
\cite{Goebbels19}, the inverse estimates of type (\ref{proto}) proposed in
\cite{Wang10} and \cite{Xu2006} are wrong. 

Similar to inverse theorems,
equivalence theorems (like (\ref{equivalence}) below)
describe equivalent behavior of expressions of moduli of smoothness and
expressions of approximation errors. Both inverse and equivalence 
theorems allow to determine smoothness properties, typically membership to Lipschitz classes or
Besov spaces, from convergence rates. 
Such a property is proved in \cite{Costarelli17} for
max-product neural network operators activated by sigmoidal functions. 
%
%
%
The relationship between order of convergence of best approximation and Besov
spaces is well understood for approximation with free knot spline functions and
rational functions, see \cite[Section 12.8]{DeVore93}, cf. \cite{LEI20131006}. The Heaviside
activation function leads to free knot splines of polynomial degree 0, i.e.,
less than $r=1$, cut and ReLU function correspond with polynomial degree less than
$r=2$. For $\sigma$ being one of these functions, and for $0<\alpha<r$, $f\in L^p[0, 1]$, $1\leq p<\infty$ ($p=\infty$
is excluded), $k:=1$
if $\alpha<1$ and $k:=2$ otherwise, $q:=\frac{1}{\alpha+1/p}$, there holds the
equivalence (see \cite{DeVore88})
\begin{equation}
 \int_0^\infty \frac{\omega_k(f,
 t)_{q}^{q}}{t^{1+\alpha q}}\,
 dt <
 \infty
 \iff
\sum_{n=1}^\infty \frac{\left(E(\Phi_{n,\sigma},
f)_p\right)^{q}}{n^{1-\alpha q}} < \infty.
\label{equivalence}
\end{equation}
However, such equivalence theorems might not be suited to obtain little-o
results: 
Assume that $E(\Phi_{n,\sigma}, f)_p =
\frac{1}{n^\beta (\ln(n+1))^{1/q}}
=o\left(\frac{1}{n^\beta}\right)$, then the right side of (\ref{equivalence})
converges exactly for the same smoothness parameters $0 < \alpha < \beta$ than
if $E(\Phi_{n,\sigma}, f)_p = \frac{1}{n^\beta} \neq o\left(\frac{1}{n^\beta}\right)$.

The
third type of sharpness results is based on counterexamples. 
The present paper follows this approach to deal with little-o effects.
Without further restrictions, counterexamples show that convergence orders 
can not be faster than stated in
terms of moduli of smoothness in (\ref{d1}),
(\ref{m2}), (\ref{m2b}) and the estimates in following Section \ref{secHigh} for some activation functions.
To obtain such counterexamples, a general theorem is introduced in Section
\ref{secUBP}.
It is applied to neural network approximation in Section \ref{secsharp}.

Unlike the counterexamples in this paper,
counterexamples that do not
focus on moduli of smoothness were recently introduced
in Almira et al.~\cite{almira2018} for 
continuous piecewise polynomial activation functions $\sigma$ with finitely
many pieces (cf. Corollary \ref{thsharpsplines} below) as well as for rational
activation
functions (that we also briefly discuss in Section \ref{secsharp}): Given an
arbitrary sequence of positive real numbers $(\varepsilon_n)_{n=1}^\infty$ with $\lim_{n\to\infty} \varepsilon_n=0$, 
a continuous counterexample $f$ is constructed such that $E(\Phi_{n,\sigma}, f) \geq \varepsilon_n$ for all
$n\in\mathbb{N}$.

\section{Higher order estimates}\label{secHigh}
In this section, two upper bounds in terms of higher
order moduli of smoothness are derived from known results. Proofs are given
for the sake of completeness. 
If, e.g., $\sigma$ is arbitrarily often differentiable on some open
interval such that $\sigma$ is no polynomial on that interval then
it is known that $E(\Phi_{n,\sigma}, p_{n-1})=0$ for all polynomials $p_{n-1}$ of degree at most $n-1$, i.e.,
$
p_{n-1}\in \Po^{n}:=\{ d_{n-1} x^{n-1}+d_{n-2}x^{n-2}+\dots+d_0:
d_0,\dots,d_{n-1}\in\mathbb{R}\},
$
see \cite[p.~157]{pinkus_1999} and (\ref{trick}) below.
Thus, upper bounds for
polynomial approximation can be used as upper bounds for neural network
approximation in connection with certain activation functions.
Due to a corollary of the classical theorem of Jackson, the best approximation
to $f\in X^p[0, 1]$, $1\leq p\leq\infty$, by algebraic polynomials is bounded by
the $r$th modulus of smoothness.
For $n\geq r$, we use Theorem 6.3 in \cite[p.~220]{DeVore93} that is stated for
the interval $[-1, 1]$. However, by applying an affine transformation of $[0,
1]$ to $[-1, 1]$, we see that there exists a constant $C$ independently of $f$
and $n$ such that
\begin{equation}
\inf \{ \| f-p_n\|_{X^p[0, 1]} :  p_n\in\Po^{n+1} \} \leq  C
\omega_r\left(f,\frac{1}{n}\right)_p.
\label{jacksonzitat}
\end{equation}

Ritter proved an estimate in terms of a first order modulus of smoothness
for approximation with nearly exponential activation functions in
\cite{Ritter99}. Due to (\ref{jacksonzitat}), Ritter's proof
can be extended in a straightforward manner to higher order moduli. The special
case of estimating by a second order modulus is discussed in \cite{Wang10}.

According to \cite{Ritter99}, a function $\sigma:\mathbb{R}\to\mathbb{R}$ is
called ``nearly exponential'' iff for each $\varepsilon>0$ there exist real numbers $a$, $b$, $c$, and $d$ such
that for all $x\in (-\infty, 0]$ 
$$  |a\sigma(bx+c)+d - e^x |<\varepsilon.
$$
The logistic function fulfills this condition with $a=1/\sigma_l(c)$, $b=1$,
$d=0$, and $c<\ln(\varepsilon)$ 
such that for $x\leq 0$ there is $0<e^x\leq
1$ and
\begin{alignat*}{1}
\lefteqn{
 |a\sigma_l(bx+c)+d - e^x | =
\left|\frac{\sigma_l(x+c)}{\sigma_l(c)}-e^x\right| =
\left|\frac{1+e^{-c}}{1+e^{-c}e^{-x}} - e^x\right|}\\ 
&= e^x \left|\frac{1+e^{-c}}{e^x+e^{-c}} -1\right|
=
e^x
\frac{1-e^{x}}{e^x+e^{-c}} 
\leq 1\cdot
\frac{1-0}{0+e^{-c}} 
= e^c <\varepsilon.
\end{alignat*}

\begin{theorem}
Let $\sigma:\mathbb{R}\to\mathbb{R}$ be a nearly exponential function, $f\in
X^p[0,1]$, $1\leq p\leq \infty$, and $r\in\mathbb{N}$.
Then, independently of $n\geq \max\{r, 2\}$ (or $n>r$) and $f$, a constant $C_r$
exists such that
\begin{equation}
E(\Phi_{n,\sigma}, f)_p \leq C_r
\omega_r\left(f,\frac{1}{n}\right)_p.\label{jacksonneu}
\end{equation}\label{directsigmoid}
\end{theorem}

\begin{proof}
Let $\varepsilon>0$.
Due to (\ref{jacksonzitat}), there exists a polynomial
$p_n\in\Po^{n+1}$ of degree at most $n$ such that
\begin{equation}
\| f-p_n\|_{X^p[0, 1]} \leq  C
\omega_r\left(f,\frac{1}{n}\right)_p + \varepsilon.
\label{jacksonalt}
\end{equation}
The Jackson estimate can be used to extend the proof given for $r=1$ in
\cite{Ritter99}: Auxiliary functions ($\alpha>0$)
$$ h_\alpha(x) := \alpha \left(1-\exp\left(-\frac{x}{\alpha}\right)\right)$$
converge to $x$ pointwise for $\alpha\to\infty$ due to the theorem of
L'Hospital.
Since $\frac{d}{dx}(h_\alpha(x)-x)=0 \Longleftrightarrow e^{-x/\alpha}=1
\Longleftrightarrow x=0$, the maximum of $|h_\alpha(x)-x|$ on $[0, 1]$ is
obtained at the endpoints 0 or 1, and convergence of $h_\alpha(x)$ to $x$ is
uniform on $[0, 1]$ for $\alpha\to\infty$. Thus $\lim_{\alpha\to\infty} p_n(
h_\alpha(x)) = p_n(x)$ uniformly on $[0, 1]$, and for the given $\varepsilon$ we
can choose $\alpha$ large enough to get
\begin{equation}
\| p_n(\cdot) - p_n(h_\alpha(\cdot))\|_{B[0, 1]} <
\varepsilon.\label{expapprox}
\end{equation}
Therefore, function $f$ is approximated by an exponential sum of type
$$ p_n(h_\alpha(x)) = \gamma_0 + \sum_{k=1}^n \gamma_k \exp\left(-\frac{k
x}{\alpha}\right) $$
within the bound $C \omega_r\left(f,n^{-1}\right) +
2\varepsilon$.
It remains to approximate the exponential sum by utilizing that $\sigma$ is
nearly exponential. 
For $1\leq k\leq n$, $\gamma_k\neq0$, there exist parameters
$a_k, b_k, c_k, d_k$ such that 
$$  \left|a_k\sigma\left(b_k\left(-\frac{k x}{\alpha}\right)+c_k\right)+d_k -
\exp\left(- \frac{kx}{\alpha}\right)
\right|<\frac{\varepsilon}{n|\gamma_k|} $$ for all $x\in[0, 1]$.
Also because $\sigma$ is nearly exponential, there exists $c\in\mathbb{R}$
with $\sigma(c)\neq 0$. Thus, a constant $\gamma$ can be expressed via
$\frac{\gamma}{\sigma(c)} \sigma(0x+c)$. 
Therefore, there exists a function $g_{n,\varepsilon}\in \Phi_{n+1, \sigma}$,
\begin{alignat*}{1}
 g_{n,\varepsilon}(x) &= \frac{\gamma_0}{\sigma(c)}
 \sigma(0x+c)+  \sum_{k=1, \gamma_k\neq 0}^n \gamma_k
 \left[a_k\sigma\left(b_k\left(-\frac{k x}{\alpha}\right)+c_k\right)+d_k\right]\\
&=  
\frac{\gamma_0+\sum_{k=1, \gamma_k\neq 0}^n \gamma_k d_k}{\sigma(c)}
\sigma(0x+c) + \sum_{k=1, \gamma_k\neq 0}^n
\gamma_k a_k\sigma\left(b_k\left(-\frac{k x}{\alpha}\right)+c_k\right)
\end{alignat*}
such that
\begin{alignat}{1}
 \lefteqn{\|  p_n(h_\alpha(\cdot))  - g_{n,\varepsilon}\|_{B[0, 1]}}\nonumber\\
 &\leq \sum_{k=1, \gamma_k\neq 0}^n |\gamma_k| \left\|
 a_k\sigma\left(b_k\left(-\frac{k x}{\alpha}\right)+c_k\right)+d_k -
\exp\left(- \frac{kx}{\alpha}\right)
 \right\|_{B[0,1]}\nonumber\\
 &\leq 
n\frac{\varepsilon}{n} = \varepsilon.\label{last}
\end{alignat}
By combining (\ref{jacksonalt}), (\ref{expapprox}) and (\ref{last}), we get
\begin{alignat*}{1}
&E(\Phi_{n + 1,\sigma}, f)_p\\
 &\leq \| f-p_n\|_{X^p[0, 1]}\!\! + \! \| p_n -
p_n(h_\alpha(\cdot))\|_{B[0, 1]}\!\! + \! \|  p_n(h_\alpha(\cdot))  -
g_{n, \varepsilon}\|_{B[0, 1]}\\
&\leq C \omega_r\left(f,\frac{1}{n}\right)_p +
3\varepsilon.
\end{alignat*}
Since $\varepsilon$ can be chosen
arbitrarily, we obtain for $n\geq 2$
\begin{equation}
 E(\Phi_{n,\sigma}, f)_p \leq  C
\omega_r\left(f,\frac{1}{n-1}\right)_p \leq C
\omega_r\left(f,\frac{2}{n}\right)_p \leq  C 2^r
\omega_r\left(f,\frac{1}{n}\right)_p.\label{verschiebung}
\end{equation}
\end{proof}

By choosing $a=1/\alpha$, $b=1$, $c=0$, and $d=1$, the ELU activation function
$\sigma_e(x)$ obviously fulfills the condition to be nearly exponential. But
its definition for $x\geq 0$ plays no role. 

Given a nearly exponential activation function, a lower bound (\ref{lb}) or an
inverse estimate (\ref{proto}) with a constant $C_r$, independent of $f$, is not
valid, see \cite{Goebbels19} for $r=2$. Such inverse estimates were proposed in
\cite{Wang10} and \cite{Xu2006}.
Functions $f_n(x):=\exp(-nx)$
fulfill $\|f_n\|_{B[0,1]}=1$ and
\begin{alignat*}{1}
(1-e^{-1})^r &=
\left|\sum_{k=0}^r (-1)^k \binom{r}{k} f_n\left(\frac{k}{n}\right)
\right| \leq \omega_r\left(f_n, \frac{1}{n}\right).
\end{alignat*}
For $x\leq 0$, one can uniformly approximate $e^x$ arbitrarily well by assigning
values to $a$, $b$, $c$ and $d$ in $a\sigma(bx+c)+d$. 
Thus, each function $f_n(x)$ can be approximated uniformly 
by $a \sigma(-nb x+c)+d$ on $[0, 1]$ such that 
$E(\Phi_{k,\sigma}, f_n)=0$, $k>1$.
Using $f_n$ with formula (\ref{proto}) implies 
\begin{equation}
(1-e^{-1})^r  \leq \frac{C_r}{n^r}[E(\Phi_{1,\sigma}, f_n) +
\|f_n\|_{B[0,1]}] \leq \frac{2 C_r}{n^r}\label{antiproto}
\end{equation}
which is obviously wrong for $n\to\infty$. The same problem occurs with
(\ref{lb}).

The ``nearly exponential'' property only fits with certain activation
functions 
but it does not require continuity. For example, let $h(x)$ be
the Dirichlet function that is one for rational and zero for irrational
numbers $x$.
Activation function $\exp(x)(1+\exp(x)h(x))$ is nowhere
continuous but nearly exponential: For $\varepsilon>0$ let $c=\ln(\varepsilon)$
and $a=\exp(-c)$, $b=1$, then for $x\leq 0$ 
$$\left| e^x- a e^{x+c}(1+e^{x+c}h(x+c))\right|
=  e^{2x+c}h(x+c) \leq e^c =\varepsilon.$$
But a bound can also be obtained from arbitrarily often differentiability.
Let $\sigma$ be arbitrarily often differentiable on some open interval such
that $\sigma$ is no polynomial on that interval. Then one can easily obtain an
estimate in terms of the $r$th modulus from the Jackson estimate (\ref{jacksonzitat}) by 
considering that polynomials of degree at most $n-1$ can be
approximated arbitrarily well by functions in $\Phi_{n,\sigma}$, see
\cite[Corollary 3.6, p.~157]{pinkus_1999}, 
cf.~\cite[Theorem 3.1]{Kurkova98}.
The idea is to approximate monomials by differential quotients of $\sigma$. This is possible since derivative
\begin{equation}
 \frac{\partial^k}{\partial b^k} \sigma(bx+c) = x^k
 \sigma^{(k)}(bx+c)\label{trick}
\end{equation}
at $b=0$ equals $\sigma^{(k)}(c) x^k$. 
Constants $\sigma^{(k)}(c)\neq 0$ can be chosen because polynomials are
excluded. The following estimate extends
\cite[Theorem 6.8, p.~176]{pinkus_1999} in the univariate case to moduli of
smoothness.

\begin{theorem}
Let $\sigma:\mathbb{R}\to\mathbb{R}$ be arbitrarily often differentiable
on some open interval in $\mathbb{R}$ and let $\sigma$ be no polynomial on that
interval, $f\in X^p[0, 1]$, $1\leq p\leq \infty$,
and $r\in\mathbb{N}$.
Then, independently of $n\geq \max\{r, 2\}$ (or $n>r$) and $f$, a constant $C_r$
exists such that
\begin{equation}
E(\Phi_{n,\sigma}, f)_p \leq C_r
\omega_r\left(f,\frac{1}{n}\right)_p.\label{jacksonneu2}
\end{equation}\label{directsigmoid2}
\end{theorem}

This theorem can be applied to $\sigma_l$ but also to $\sigma_a$ and $\sigma_e$.
The preliminaries are also fulfilled for $\sigma(x):=\sin(x)$, a function that
is obviously not nearly exponential.

\begin{proof}
Let $\varepsilon>0$.
As in the previous proof, there exists a polynomial
$p_n$ of degree at most $n$ such that (\ref{jacksonalt}) holds. 
Due to \cite[p.~157]{pinkus_1999} there exists a function
$g_{n,\varepsilon}\in \Phi_{n+1,\sigma}$ such that  $\|g_{n,\varepsilon} -
p_n\|_{B[0,1]}<\varepsilon$. This gives
$$
E(\Phi_{n + 1,\sigma}, f)_p \leq \| f-p_n\|_{X^p[0, 1]} +  \| p_n -
g_{n,\varepsilon}\|_{B[0, 1]}
\leq C \omega_r\left(f,\frac{1}{n}\right)_p +
2\varepsilon.
$$
Since $\varepsilon$ can be chosen arbitrarily, we get (\ref{jacksonneu2}) via
equation (\ref{verschiebung}).
\end{proof}
%
Polynomials in the closure of approximation spaces can be utilized to show that
a direct lower bound in terms of a (uniform) modulus of smoothness is not possible.
\begin{lemma}[Impossible Inverse Estimate]\label{lower}
Let activation function $\varphi$ be given as in the preceding
Theorem \ref{directsigmoid2}, $r\in\mathbb{N}$. For each
positive, monotonically decreasing sequence $(\alpha_n)_{n=1}^\infty$,
$\alpha_n>0$, and each $0<\beta<1$ a counterexample $f_\beta\in
C[0, 1]$ exists such that (for $n\to\infty$)
\begin{alignat}{1}
&\omega_r\left(f_\beta,\frac{1}{n}\right) \leq C
\frac{1}{n^{r\beta}},\quad \omega_r\left(f_\beta,\frac{1}{n}\right) \neq
o\left(\frac{1}{n^{r\beta}}\right),\label{auslip}\\
&\limsup_{n\to\infty}
\frac{\omega_r\left(f_\beta,\frac{1}{n}\right)}{E(\Phi_{n,\sigma},
f_\beta)}\alpha_n > 0\text{, i.e., }
\omega_r\left(f_\beta,\frac{1}{n}\right) \neq o\left(\frac{1}{\alpha_n}
E(\Phi_{n,\sigma}, f_\beta)\right).
\label{gehtnicht}
\end{alignat}
\end{lemma}
Even if the constant $C=C_f>0$ may depend on $f$ (but not on $n$), 
estimate (\ref{lb}), 
as proposed in a similar context in \cite{Wang10} for $r=2$, does not apply.
\begin{proof}
For best trigonometric approximation, comparable counter
examples are constructed in \cite[Corollary 3.1]{Dickmeis84} based on a general
resonance theorem \cite[Theorem 2.1]{Dickmeis84}. We apply this theorem to show
a similar result for best algebraic polynomial approximation, i.e., there exists
$f_\beta\in C[0, 1]$ such that (\ref{auslip}) and
\begin{equation} 
\limsup_{n\to\infty}
\frac{\omega_r\left(f_\beta,\frac{1}{n}\right)}{E(\Pi^{n+1},
f_\beta)}\alpha_{n+1} > 0.\label{lowerp}
\end{equation}
We choose parameters of \cite[Theorem 2.1]{Dickmeis84} as follows:
$X=C[0, 1]$, $\Lambda=\{r\}$, $T_n(f) =
S_{n,r}(f):= \omega_r\left(f,\frac{1}{n}\right)$ 
such that
$C_2=C_{5,n}=2^r$, $\omega(\delta)=\delta^\beta$. 
Different to \cite[Corollary
3.1]{Dickmeis84} we use resonance functions $h_j(x):=x^j$, $j\in\mathbb{N}$, such that $C_1=1$.
We further set
$\varphi_n=\tau_n=\psi_n=1/n^r$ such that with (\ref{gegenAbl})
$$|T_n(h_j)| \leq \frac{1}{n^r}\|h_j^{(r)}\|_{B[0, 1]} \leq 
\frac{\varphi_n}{\varphi_j},\quad
 S_{n,r}(h_j) \leq \frac{\varphi_n}{\varphi_j} =:
C_{6,j}\varphi_n=C_{6,j}\tau_n. $$
We compute an $r$-th difference of $x^n$ at the interval endpoint 1 to get
the resonance condition
\begin{alignat*}{1}
&\liminf_{n\to\infty}|S_{n,r}(h_n)| =
\liminf_{n\to\infty} \omega_r\left(x^n,\frac{1}{n}\right) \geq
\lim_{n\to\infty} \sum_{k=0}^r (-1)^k \binom{r}{k}
\left(1-\frac{k}{n}\right)^n\\
&= \sum_{k=0}^r (-1)^k \binom{r}{k} \exp(-k) = \sum_{k=0}^r \binom{r}{k}
\left(-\frac{1}{e}\right)^k =  \left(1-\frac{1}{e}\right)^r =: C_{7,r}>0.
\end{alignat*}
Further, let $R_n(f):=  E(\Pi^{n+1}, f)$ such that $R_n(h_j)=0$ for $j\leq n$, i.e., 
$0=R_n(h_j)\leq C_4\frac{\rho_n}{\psi_n}$ for
$\rho_n:=\alpha_{n+1}\psi_n$.
Then, due to \cite[Theorem 2.1]{Dickmeis84}, $f_\beta\in C[0, 1]$ exists such
that (\ref{auslip}) and (\ref{lowerp}) are fulfilled.

Since polynomials of $\Pi^{n+1}$ are in the closure of $\Phi_{n+1,\sigma}$
according to \cite[p.~157]{pinkus_1999}, i.e., 
$E(\Phi_{n+1,\sigma}, f_\beta) \leq E(\Pi^{n+1}, f_\beta)$, we obtain
(\ref{gehtnicht}) from (\ref{lowerp}):
\begin{alignat*}{1}
\limsup_{n\to\infty}
\frac{\omega_r\left(f_\beta,\frac{1}{n+1}\right)}{E(\Phi_{n+1,\sigma},
f_\beta)}\alpha_{n+1} &\geq
\limsup_{n\to\infty}
\frac{\frac{1}{2^r}
\omega_r\left(f_\beta,\frac{1}{n}\right)}{E(\Phi_{n+1,\sigma},
f_\beta)}\alpha_{n+1}\\ 
&\geq \limsup_{n\to\infty} \frac{1}{2^r}
\frac{\omega_r\left(f_\beta,\frac{1}{n}\right)}{E(\Pi^{n+1},
f_\beta)}\alpha_{n+1} > 0.
\end{alignat*}
\end{proof}

\section{A Uniform Boundedness Principle with Rates}\label{secUBP}

In this paper, sharpness results are proved with a quantitative
extension of the classical uniform boundedness principle of Functional Analysis. 
Dickmeis, Nessel and van Wickern developed several versions of such theorems. 
We already used one of them in the proof of Lemma \ref{lower}. 
An
overview of applications in Numerical Analysis can be found in
\cite[Section 6]{Goebbels13b}.
The given paper is based on \cite[p.~108]{DiNeWi2}.
This and most other versions require error functionals to be sub-additive.
Let $X$ be a normed space. A functional $T$ on $X$, i.e., $T$ maps $X$ into
$\mathbb{R}$, is said to be (non-negative-valued) sub-linear and bounded, iff
for all $f,g\in X,\, c\in\mathbb{R}$
\begin{alignat*}{3}
 &T(f) \geq 0,~ T(f+g) \leq T(f) + T(g) &\quad&\text{(sub-additivity)}, \\
 &T(c f) = |c|T(f) &&\text{(absolute homogeneity)},\\
 &\| T\|_{X^\sim} := \sup\{ T(f) : \|f\|_X\leq 1\} < \infty &&\text{(bounded
 functional)}.
\end{alignat*}
The set of non-negative-valued
sub-linear bounded functionals $T$ on $X$ is denoted by
$X^\sim$.
Typically, errors of best approximation are (non-negative-valued) sub-linear
bounded functionals.
Let $U\subset X$ be a linear subspace. The best approximation of $f\in X$ by
elements $u\in U\neq\emptyset$ is defined as $E(f):=\inf\{ \|f-u\|_X :
u\in U\}$. Then $E$ is sub-linear: $E(f+g)\leq E(f)+E(g)$, $E(cf) = |c|E(f)$ for
all $c\in\mathbb{R}$. Also, $E$ is bounded: $E(f)\leq \|f-0\|_X=\|f\|_X$.

Unfortunately, function sets $\Phi_{n,\sigma}$ are not linear spaces,
cf.\ \cite[p.~151]{pinkus_1999}. In general, from $f,g\in \Phi_{n,\sigma}$ one
can only conclude $f+g\in \Phi_{2n,\sigma}$ whereas $cf\in \Phi_{n,\sigma}$, $c\in\mathbb{R}$.
Functionals of best approximation fulfill $E(\Phi_{n,\sigma}, f)_p \leq
\|f-0\|_{X^p[0, 1]}=\|f\|_{X^p[0, 1]}$. Absolute homogeneity $E(\Phi_{n,\sigma},
cf)_p =|c|E(\Phi_{n,\sigma}, f)_p$ is obvious for $c=0$. If $c\neq 0$,
\begin{alignat*}{1}
E(\Phi_{n,\sigma}, cf)_p &= \inf\left\{ \left\|cf - \sum_{k=1}^n a_k
\sigma(b_k x+c_k)\right\|_{X^p[0, 1]} :
a_k,b_k,c_k\in\mathbb{R}\right\}\\ 
&= |c| \inf\left\{ \left\| f - \sum_{k=1}^n \frac{a_k}{c}
\sigma(b_k x+c_k)\right\|_{X^p[0, 1]} :
\frac{a_k}{c},b_k,c_k\in\mathbb{R}\right\}\\ 
&=|c|
E(\Phi_{n,\sigma}, f)_p.
\end{alignat*}

But there is no sub-additivity. However, it is easy
to prove a similar condition:
For each $\varepsilon>0$ there exists elements
$u_{f,\varepsilon},u_{g,\varepsilon}\in \Phi_{n,\sigma}$ that fulfill
$$ \|f-u_{f,\varepsilon}\|_{X^p[0, 1]}\leq
E(\Phi_{n,\sigma}, f)_p+\frac{\varepsilon}{2},\quad \|
g-u_{g,\varepsilon}\|_{X^p[0, 1]}\leq E(\Phi_{n,\sigma},
g)_p+\frac{\varepsilon}{2} $$ and $u_{f,\varepsilon}+u_{g,\varepsilon}\in
\Phi_{2n,\sigma}$ such that
\begin{alignat*}{1}
 \lefteqn{E(\Phi_{2n,\sigma}, f+g)_p \leq \|f-u_{f,\varepsilon} +
g-u_{g,\varepsilon}\|_{X^p[0, 1]}}\\ 
&\leq \|f-u_{f,\varepsilon}\|_{X^p[0, 1]} +
\|g-u_{g,\varepsilon}\|_{X^p[0, 1]} \leq E(\Phi_{n,\sigma},
f)_p+E(\Phi_{n,\sigma}, g)_p+\varepsilon, 
\end{alignat*} 
i.e., 
\begin{equation}
E(\Phi_{2n,\sigma}, f+g)_p \leq
E(\Phi_{n,\sigma}, f)_p+E(\Phi_{n,\sigma}, g)_p.\label{nonsublin}
\end{equation}
Obviously, also $E(\Phi_{n,\sigma}, f)_p \geq E(\Phi_{n+1,\sigma}, f)_p$ holds
true.

In what follows, a quantitative extension of the uniform boundedness
principle based on these conditions is presented. The conditions replace
sub-ad\-di\-tivi\-ty.
Another extension of the uniform boundedness principle to non-sub-linear functionals is
proved in \cite{Dickmeis1985}.
But this version of the theorem is stated for a family of error functionals
with two parameters that has to fulfill a condition of quasi lower
semi-continuity. Functionals $S_\delta$ measuring smoothness also do not
need to be sub-additive but have to fulfill a condition $S_\delta(f+g)\leq
B(S_\delta(f)+S_\delta(g))$ for a constant $B\geq 1$. This theorem does not
consider replacement (\ref{nonsublin}) for sub-additivity.

Both rate of convergence and size of moduli of smoothness can be expressed by
abstract moduli of smoothness, see \cite[p.~96ff]{Timan63}.
Such an abstract modulus of smoothness is defined as a continuous, increasing
function $\omega$ on $[0,\infty)$ 
that fulfills
\begin{equation}
 0=\omega(0)<\omega(\delta_1)\leq\omega(\delta_1+\delta_2)
 \leq\omega(\delta_1)+\omega(\delta_2) \label{ABSMOD}
\end{equation}
for
all $\delta_1,\delta_2>0$. 
It has similar
properties as $\omega_r (f, \cdot)$, and
it directly follows for $\lambda>0$ that
\begin{equation}
\omega(\lambda \delta) \leq \omega(\lceil \lambda\rceil \delta) \leq \lceil
\lambda\rceil \omega(\delta) \leq (\lambda+1)  \omega(\delta).\label{lomega}
\end{equation}
Due to continuity, $\lim_{\delta\to 0+} \omega(\delta)=0$ holds.
For all $0<\delta_1\leq\delta_2$, equation (\ref{lomega}) also implies
\begin{equation}
\frac{\omega(\delta_2)}{\delta_2} 
= \frac{\omega\left(\frac{\delta_2}{\delta_1}\delta_1\right)}{\delta_2}
\leq \frac{\left(1+\frac{\delta_2}{\delta_1}\right) \omega(\delta_1)}{\delta_2}
= \frac{\left(\frac{\delta_1}{\delta_2}+1\right) \omega(\delta_1)}{\delta_1}
\leq 2
\frac{\omega(\delta_1)}{\delta_1}.\label{la2c}
\end{equation}

Functions $\omega(\delta):=\delta^\alpha$,
$0 < \alpha\leq 1$, are examples for abstract moduli of smoothness. They are
used to define Lipschitz classes.

The aim is to discuss a sequence of remainders (that will be errors of best
approximation) $(E_{n})_{n=1}^\infty$, $E_n : X\to [0,\infty)$. These functionals
do not have to be sub-linear but instead have to fulfill
\begin{alignat}{1}
E_{m\cdot n}\left(\sum_{k=1}^m f_k\right) &\leq \sum_{k=1}^m
E_n(f_k) \text{ (cf. (\ref{nonsublin}))}\label{semisub}\\
E_n(cf) &= |c| E_n(f)\label{semilin}\\
E_n(f) &\leq D_n \|f\|_X\label{bounded}\\
E_n(f) &\geq E_{n+1}(f)\label{mono}
\end{alignat}
for all $m\in\mathbb{N}$, $f, f_1, f_2,\dots, f_m\in X$, and constants
$c\in\mathbb{R}$.
In the boundedness condition (\ref{bounded}), $D_n$ is a constant only depending on $E_n$ but not on $f$.

\begin{theorem}[Adapted Uniform Boundedness Principle]\label{ZUBP} Let
$X$ be a (real) Banach space with norm $\|\cdot\|_X$.
Also, a sequence $(E_{n})_{n=1}^\infty$, $E_n : X\to [0,\infty)$ is given that fulfills
conditions (\ref{semisub})--(\ref{mono}).
To measure smoothness, sub-linear bounded functionals $S_\delta \in
X^\sim$ are used for all $\delta>0$.

Let $\mu(\delta): (0,\infty)\to  (0,\infty)$ be a positive function, and 
let $\varphi: [1,\infty)\to  (0,\infty)$ be a strictly decreasing
function with $\lim_{x\to\infty} \varphi(x)=0$. An additional requirement is
that for each $0<\lambda<1$ a point $X_0=X_0(\lambda)\geq \lambda^{-1}$ and
constant $C_\lambda>0$ exist such that
\begin{equation}
\varphi(\lambda x)\leq
C_\lambda \varphi(x)\label{faktor}
\end{equation}
for all $x>X_0$. 

If there exist test elements $h_n\in X$ such that for all
$n\in\mathbb{N}$ with $n\geq n_0\in\mathbb{N}$ and $\delta>0$
\begin{alignat}{1}
 \| h_n\|_X &\leq C_1,
           \label{ZUBP1}\\
  S_\delta(h_n) &\leq C_2 \min \left\{ 1,\frac{\mu(\delta)}{\varphi(n)}
           \right\},
           \label{ZZUBP3}\\
  E_{4n}(h_n) &\geq c_{3} >0,\label{ZUBP3}
\end{alignat}
then for each abstract modulus of smoothness $\omega$ satisfying 
\begin{equation}
  \lim_{\delta\to 0+}\frac{\omega(\delta)}{\delta}=\infty\label{odd}
\end{equation}
there exists a counterexample $f_\omega\in
X$ such that $(\delta\to 0+,\, n\to\infty)$
\begin{alignat}{1}
 S_\delta(f_\omega) &= {O}\left(\omega(\mu(\delta))\right),\label{smooth}\\
 E_{n}(f_\omega) &\not= o(\omega(\varphi(n)))\text{, i.e., }
 \limsup_{n\to\infty} \frac{E_{n}(f_\omega)}{\omega(\varphi(n))} > 0.\label{counter}
\end{alignat}
\end{theorem}
For example, (\ref{faktor}) is fulfilled for a standard choice
$\varphi(x)=1/x^\alpha$.

The prerequisites of the theorem differ from the Theorems of Dickmeis, Nessel,
and van Wickern in conditions (\ref{semisub})--(\ref{mono}) that replace
$E_n\in X^\sim$. It also requires additional constraint (\ref{faktor}). For
convenience, resonance condition (\ref{ZUBP3}) replaces $E_{n}(h_n) \geq c_{3}$.
Without much effort, (\ref{ZUBP3}) can be weakened to $\limsup_{n\to\infty} 
E_{4n}(h_n) > 0$. 

The proof is based on a gliding hump and follows the ideas
of \cite[Section 2.2]{DiNeWi2} (cf.\ \cite{Dickmeis1984}) for sub-linear
functionals
and the literature cited there. For the sake of completeness, the
whole proof is presented although changes were required only for estimates that
are effected by missing sub-additivity.

\begin{proof}
The first part of the proof is not concerned with sub-additivity or its
replacement.
If a test element $h_j$, 
$j\geq n_0$, exists that already fulfills
\begin{equation}
 \limsup_{n\to\infty} \frac{E_n(h_j)}{\omega(\varphi(n))} >0,\label{condsimple}
\end{equation}
then $f_\omega:=h_j$ fulfills (\ref{counter}). 
To show that this  $f_\omega$ also fulfills (\ref{smooth}), one needs
inequality
\begin{equation}
\min\{ 1,\delta\} \leq A\omega(\delta)
\label{omegaab}
\end{equation}
for all $\delta>0$.
This inequality follows from (\ref{la2c}): If $0<\delta<1$ then 
$\omega(1)/1 \leq 2 \omega(\delta)/\delta,
$ such that $\delta\leq 2\omega(\delta)/\omega(1)$. If $\delta>1$ then
$\omega(1)\leq\omega(\delta)$, i.e., $1 \leq \omega(\delta)/\omega(1)$.
Thus, one can choose $A=2/\omega(1)$.

Smoothness (\ref{smooth}) of test elements $h_j$ now follows from
(\ref{omegaab}):
\begin{alignat*}{1}
S_\delta(h_j) &\stackrel{\text{(\ref{ZZUBP3})}}{\leq} C_2  \min \left\{
1,\frac{\mu(\delta)}{\varphi(j)} \right\} 
\stackrel{\text{(\ref{omegaab})}}{\leq}
AC_2\omega\left(\frac{\mu(\delta)}{\varphi(j)}\right)\\ 
&\stackrel{\text{(\ref{lomega})}}{\leq}  AC_2
\left(1+\frac{1}{\varphi(j)}\right) \omega(\mu(\delta)).
\end{alignat*}
Under condition (\ref{condsimple}) function $f_\omega:=h_j$ indeed is a counter
example.
Thus, for the remaining proof one can assume that for all $j\in\mathbb{N}$,
$j\geq n_0$:
\begin{equation}
 \lim_{n\to\infty} \frac{E_n(h_j)}{\omega(\varphi(n))} = 0.\label{null}
\end{equation}

The arguments of Dickmeis, Nessel and van
Wickern have to be adjusted to missing sub-additivity in the next part of the
proof. It has to be shown that for each fixed $m\in\mathbb{N}$ a
finite sum inherits limit (\ref{null}). Let $(a_l)_{l=1}^m\subset\mathbb{R}$ and
$j_1,\dots,j_m \geq n_0$ different indices. To prove
\begin{equation}
\lim_{n\to\infty} \frac{E_n\left(\sum_{l=1}^m a_l h_{j_l}\right)
}{\omega(\varphi(n))} = 0,\label{eq2}
\end{equation}
one can apply (\ref{mono}), (\ref{semisub}), and (\ref{semilin}) for
$n\geq 2m$:
\begin{alignat}{1}
0 &\leq \frac{E_n\left(\sum_{l=1}^m a_l h_{j_l}\right)
}{\omega(\varphi(n))} \stackrel{\text{(\ref{mono})}}{\leq} \frac{E_{m\cdot
\lfloor n/m \rfloor}\left(\sum_{l=1}^m a_l h_{j_l}\right) }{\omega(\varphi(n))}
 \stackrel{\text{(\ref{semisub})}}{\leq} \frac{\sum_{l=1}^m E_{\lfloor n/m
\rfloor}\left(a_l h_{j_l}\right) }{\omega(\varphi(n))}\nonumber\\
 &\stackrel{\text{(\ref{semilin})}}{=} \sum_{l=1}^m |a_l| \frac{E_{\lfloor
 n/m \rfloor}\left( h_{j_l}\right)
 }{\omega(\varphi(n))}.\label{zwischen}
\end{alignat}
Since $\varphi(x)$ is decreasing and $\omega(\delta)$ is increasing,
$\omega(\varphi(x))$ is decreasing. Thus,
(\ref{faktor}) for $\lambda:=(2m)^{-1}$ and $n>\max\{2m, X_0(\lambda)\}$ implies
\begin{alignat*}{1}
\omega\left(\varphi\left(\lfloor n/m
\rfloor\right)\right)
&\leq
\omega\left(\varphi\left(\frac{n-m}{m}\right)\right)
 \leq
 \omega\left(\varphi\left(\frac{n-n/2}{m}\right)\right)
 \stackrel{\text{(\ref{faktor})}}{\leq} \omega\left(C_{\frac{1}{2m}}\varphi( n
 )\right)\\
&\stackrel{\text{(\ref{lomega})}}{\leq} \lceil C_{\frac{1}{2m}}\rceil
\omega\left(\varphi( n )\right).
\end{alignat*}
With this inequality, estimate (\ref{zwischen}) becomes
$$
0 \leq \frac{E_n\left(\sum_{l=1}^m a_l h_{j_l}\right)
}{\omega(\varphi(n))}
 \leq \lceil C_{\frac{1}{2m}} \rceil \sum_{l=1}^m |a_l| \frac{ E_{\lfloor n/m
\rfloor}\left(
 h_{j_l}\right)
 }{\omega\left(\varphi\left(\lfloor n/m
\rfloor\right)\right)}.
$$
According to (\ref{null}), this gives (\ref{eq2}).

Now one can select a sequence $(n_k)_{k=1}^\infty\subset\mathbb{N}$, $n_0 \leq n_k<n_{k+1}$ for all $k\in\mathbb{N}$, to construct a suitable
counterexample
\begin{equation}
f_\omega := \sum_{k=1}^\infty \omega(\varphi(n_k)) h_{n_k}.\label{deffw}
\end{equation} 

Let $n_1:=n_0$. If $n_1,\dots,n_k$ have already be chosen then select
$n_{k+1}$ with
$n_{k+1}\geq 2k$ and
$n_{k+1}> n_k$
large enough to fulfill following conditions:
\begin{alignat}{3}
&\omega(\varphi(n_{k+1})) \leq \frac{1}{2}
\omega(\varphi(n_k)) &\quad&\text{$\left(\lim_{x\to\infty}
\omega(\varphi(x))=0\right)$}\label{eq3}\\
&D_{2\cdot n_k} \omega(\varphi(n_{k+1})) \leq \frac{
\omega(\varphi(n_{k}))}{k} &&\text{$\left(\lim_{x\to\infty}
\omega(\varphi(x))=0\right)$}\label{eq5}\\
&\sum_{j=1}^k \frac{\omega(\varphi(n_j))}{\varphi(n_j)} \leq
\frac{\omega(\varphi(n_{k+1}))}{\varphi(n_{k+1})} &&\text{(cf.
(\ref{odd}))}\label{eq4}\\
&\frac{E_{n_{k+1}}\left(\sum_{l=1}^k \omega(\varphi(n_l)) h_{n_l}\right)
}{\omega(\varphi(n_{k+1}))} \leq \frac{1}{k+1}
 &&\text{(cf. (\ref{eq2}))}.\label{eq6}
\end{alignat}
Only condition
(\ref{eq5}) is adjusted to missing sub-additivity.
The next part of the proof does not consider properties of $E_n$, see \cite{DiNeWi2}.

Function $f_\omega$ in (\ref{deffw}) is well-defined:
For $j\geq k$, iterative application of (\ref{eq3}) leads to
$$ \omega(\varphi(n_{j})) \leq 2^{-1}  \omega(\varphi(n_{j-1})) \leq\cdots\leq 
2^{k-j} \omega(\varphi(n_{k})). $$
This implies
\begin{equation}
\sum_{j=k}^\infty \omega(\varphi(n_j)) 
\leq \sum_{j=k}^\infty 2^{k-j} \omega(\varphi(n_k)) 
= 2
\omega(\varphi(n_k)).\label{restsum}
\end{equation}
With this estimate (and because of $\lim_{k\to\infty}\omega(\varphi(n_k))=0$),
it is easy to see that $(g_m)_{m=1}^\infty$, $g_m:=\sum_{j=1}^m
\omega(\varphi(n_i)) h_{n_i}$, is a Cauchy sequence that converges to $f_\omega$ in Banach space $X$: For a given $\varepsilon>0$, there
exists a number $N_0(\varepsilon)>n_0$ such that
$\omega(\varphi(n_k))<\varepsilon/(2C_1)$ for all $k>N_0$. Then, due to
(\ref{ZUBP1}) and 
(\ref{restsum}), for all $k >i >N_0$:
$$ \| g_k-g_i \|_X \leq \sum_{j=i+1}^k\! \omega(\varphi(n_j)) \| h_{n_j}\|_X 
\leq  C_1 \sum_{j=i+1}^\infty\! \omega(\varphi(n_j))
\leq
2C_1 \omega(\varphi(n_{i+1})) < \varepsilon.$$ 
Thus, the Banach condition is fulfilled and counterexample $f_\omega\in X$ is
well defined.

Smoothness condition (\ref{ZZUBP3}) is proved in two cases. The first
case covers numbers $\delta>0$ for which 
$\mu(\delta)\leq \varphi(n_1)$.
Since $\lim_{x\to\infty} \varphi(x)=0$, there exists $k\in\mathbb{N}$ such that in this case  
\begin{equation}
\varphi(n_{k+1}) \leq \mu(\delta) \leq \varphi(n_k).\label{neuIntervall}
\end{equation}
Using this index $k$ in connection with the two bounds in (\ref{ZZUBP3}), one
obtains for sub-linear functional $S_\delta$
\begin{eqnarray}
\lefteqn{S_\delta(f_\omega) \leq \sum_{j=1}^k  \omega(\varphi(n_j))
S_\delta(h_{n_j}) +  \sum_{j=k+1}^\infty  \omega(\varphi(n_j))
S_\delta(h_{n_j})}\label{seso}\\
&\stackrel{\text{(\ref{ZZUBP3})}}{\leq}&
C_2 \left[ \left(\sum_{j=1}^{k-1}  \omega(\varphi(n_j))
\frac{\mu(\delta)}{\varphi(n_j)}\right) +  \omega(\varphi(n_k))
\frac{\mu(\delta)}{\varphi(n_k)} + \sum_{j=k+1}^\infty 
\omega(\varphi(n_j))\right]\nonumber\\
&\!\!\!\!\stackrel{\text{(\ref{eq4}), (\ref{restsum})}}{\leq}\!\!\!\!&
2 C_2 \mu(\delta) \frac{\omega(\varphi(n_k))}{\varphi(n_k)} +  2C_2
\omega(\varphi(n_{k+1}))\nonumber\\
&\leq& 2 C_2 \mu(\delta) \frac{\omega(\varphi(n_k))}{\varphi(n_k)} + 2C_2
\omega(\mu(\delta)).\label{sda}
\end{eqnarray}
The last estimate holds true because $\varphi(n_{k+1}) \leq \mu(\delta)$ in
(\ref{neuIntervall}).
The first expression in (\ref{sda}) can be estimated by $4C_2
\omega(\mu(\delta))$ with (\ref{neuIntervall}): Because $\mu(\delta)\leq
\varphi(n_k)$, one can apply (\ref{la2c}) to obtain
$$
\frac{\omega(\varphi(n_k))}{\varphi(n_k)} \leq 2 \frac{\omega(\mu(\delta))}{\mu(\delta)}. $$
Thus, $S_\delta(f_\omega) \leq  6C_2 \omega(\mu(\delta))$.

The second case is $\mu(\delta) > \varphi(n_1)$. In this situation, let
$k:=0$. Then only the second sum in (\ref{seso}) has to be considered:
$S_\delta(f_\omega) \leq  2C_2 \omega(\varphi(n_1)) \leq  2C_2
\omega(\mu(\delta))$.

The little-o condition remains to be proven without sub-additivity. 
From (\ref{semisub}) one obtains $E_{2n}(f) = E_{2n}(f+g-g) \leq
E_{n}(f+g)+E_{n}(-g)$, i.e.,
\begin{equation}
  E_{n}(f+g) \geq E_{2n}(f) - E_{n}(-g) \stackrel{\text{(\ref{semilin})}}{=}
E_{2n}(f) - E_{n}(g).\label{abres}
\end{equation}
The estimate can be used to show the desired lower bound based on
resonance condition (\ref{ZUBP3})
with the gliding hump method. Functional $E_{n_k}$ is in resonance with summand
$\omega(\varphi(n_k)) h_{n_k}$. This part of the sum forms the hump.
\begin{alignat*}{1}
\lefteqn{E_{n_k}(f_\omega) =
E_{n_k}\left(\omega(\varphi(n_k)) h_{n_k} + \sum_{j=1}^{k-1}\omega(\varphi(n_j))
h_{n_j} +  \sum_{j=k+1}^{\infty}\omega(\varphi(n_j))
h_{n_j}\right)}\\
&\stackrel{\text{(\ref{abres})}}{\geq} E_{2\cdot
n_k}\left(\omega(\varphi(n_k)) h_{n_k} + \sum_{j=k+1}^{\infty}\omega(\varphi(n_j)) h_{n_j}\right) -
E_{n_k}\left(\sum_{j=1}^{k-1}\omega(\varphi(n_j)) h_{n_j}\right)\\
&\stackrel{\text{(\ref{abres}), (\ref{semilin})}}{\geq} \omega(\varphi(n_k))
E_{4\cdot n_k}\left( h_{n_k}\right) - E_{2\cdot n_k}\left(
\sum_{j=k+1}^{\infty}\omega(\varphi(n_j)) h_{n_j}\right)\\
&\qquad -
E_{n_k}\left(\sum_{j=1}^{k-1}\omega(\varphi(n_j)) h_{n_j}\right)\\
&\stackrel{\text{(\ref{ZUBP3}), (\ref{bounded}), (\ref{eq6})}}{\geq}
\omega(\varphi(n_k)) c_3 -D_{2n_k}
\left(\sum_{j=k+1}^{\infty}\omega(\varphi(n_j))\left\| h_{n_j} \right\|_X\right)
- \frac{\omega(\varphi(n_k))}{k}\\
&\stackrel{\text{(\ref{ZUBP1}), (\ref{restsum})}}{\geq} \omega(\varphi(n_k))
c_3 -D_{2n_k} C_1 2\omega(\varphi(n_{k+1})) -
\frac{\omega(\varphi(n_k))}{k}\nonumber\\
&\stackrel{\text{(\ref{eq5})}}{\geq} \left(c_3
- \frac{2C_1+1}{k}\right) \omega(\varphi(n_k)).
\end{alignat*}
Thus $E_{n}(f_\omega) \neq o(\omega(\varphi(n)))$.
\end{proof}

\section{Sharpness}\label{secsharp}

Free knot spline function approximations by Heaviside, cut and ReLU functions
are first examples for application of Theorem \ref{ZUBP}.

Let $S_n^r$ be the space of functions $f$ for which $n+1$ intervals $]x_k,
x_{k+1}[$, $0=x_0 < x_1 < \dots < x_{n+1}=1$, exist such that $f$ equals 
(potentially different) polynomials $p$ of degree less than $r$ on each of these
intervals, i.e.\ $p\in \Por$.
No additional smoothness conditions are required at knots.

\begin{corollary}[Free Knot Spline Approximation]\label{thsharpsplines}
For $r, \tilde{r}\in\mathbb{N}$, $1\leq p\leq\infty$, and for each abstract
modulus of smoothness $\omega$ satisfying (\ref{odd}),
there exists a counterexample
$f_{\omega}\in X^p[0,1]$ such that 
\begin{alignat*}{1}
 \omega_r(f_{\omega}, \delta)_p &= {O}\left(\omega(\delta^r)\right),\\
  E(S_n^{\tilde{r}}, f_{\omega})_p &:=\inf\{\|f_{\omega}-g\|_{X^p[0, 1]} : g\in
  S_n^{\tilde{r}}\} \neq o\left(\omega\left(\frac{1}{n^r}\right)\right).
\end{alignat*}
\end{corollary}
Note that $r$ and $\tilde{r}$ can be chosen independently. This corresponds with
Marchaud inequality for moduli of smoothness.

The following lemma helps in the proof of this and the next corollary. It is
used to show the resonance condition of Theorem \ref{ZUBP}.
\begin{lemma}\label{laresonanz}
Let $g:[0, 1] \to\mathbb{R}$, and $0=x_0 < x_1 < \dots < x_{N+1}=1$. 
Assume that for each interval $I_k:=(x_k, x_{k+1})$, $0\leq k\leq N$, either
$g(x)\geq 0$ for all $x\in I_k$ or $g(x)\leq 0$ for all $x\in I_k$ holds. Then
$g$ can change its sign only at points $x_k$.
Let $h(x):=\sin(2N\cdot 2\pi\cdot x)$. Then there exists a constant $c>0$ that
is independent of $g$ and $N$ such that 
$$  \| h - g\|_{X^p[0, 1]} \geq c > 0.
$$ 
\end{lemma}
The prerequisites on $g$ are fulfilled if $g$ is continuous with at most $N$
zeroes.
\begin{proof}
We discuss $2N$ intervals $A_k:=(k (2N)^{-1}, (k+1) (2N)^{-1})$, $0\leq k<2N$.
Function $g$ can change its sign at most in $N$ of these intervals.
Let $J\subset\{0,1,\dots,2N-1\}$ the set of indices $k$ of the at least
$N$ intervals $A_k$ on which $g$ is non-negative or non-positive.
On each of these intervals, $h$ maps to both its maximum 1 and
its minimum $-1$. Thus $\|h-g\|_{B[0,1]}\geq 1$. This shows the Lemma for
$p=\infty$. 
Functions $h$ and $g$ have different sign on $(a, b)$ where
$(a, b)=(k (2N)^{-1}, (k+1/2) (2N)^{-1})$ or $(a, b)=((k+1/2) (2N)^{-1},
(k+1) (2N)^{-1})$, $k\in J$.
For all $x\in(a, b)$ we see that
$|h(x)-g(x)|\geq |h(x)| = \sin(2N\cdot 2\pi\cdot (x-a))$.
Thus, for $1\leq p<\infty$,
\begin{alignat*}{1}
\lefteqn{
\| h-g\|_{L^p[0,1]} \geq \left[\sum_{k\in J}\int_{A_k}
|h-g|^p\right]^{\frac{1}{p}}}\\
&\geq \left[ N\int_{0}^{\frac{1}{4N}}  \sin^p\left(2N\cdot 2\pi\cdot
x\right)\, dx\right]^{\frac{1}{p}} 
\geq \left[ \frac{N}{N 4\pi}\int_{0}^{\pi} 
\sin^p(u)\, du\right]^{\frac{1}{p}} =: c > 0.
\end{alignat*}
\end{proof}

\begin{proof} (of Corollary \ref{thsharpsplines})
Theorem \ref{ZUBP} can be applied with
following parameters.
Let Banach-space $X=X^p[0, 1]$.  
$$ E_n(f):= E(S_n^{\tilde{r}}, f)_p, \quad S_\delta(f) := \omega_r(f, \delta)_p.
$$
Whereas $S_\delta$ is a sub-linear, bounded functional, errors of best
approximation $E_n$ fulfill conditions (\ref{semisub}), (\ref{semilin}),
(\ref{bounded}), and (\ref{mono}), cf. (\ref{nonsublin}),
with $D_n=1$.
Let $\mu(\delta):=\delta^r$ and $\varphi(x)=1/x^r$ such that condition
(\ref{faktor}) holds: $\varphi(\lambda x)=\varphi(x)/\lambda^r$.
Let $N=N(n):= (4n+1)\tilde{r}$. Resonance elements
$$
h_n(x):=\sin(2N\cdot 2\pi \cdot x)=\sin((16n+4)\tilde{r} \pi \cdot x)
$$
obviously satisfy condition (\ref{ZUBP1}): $\| h_n(x)\|_{X^p[0, 1]}\leq 1=:C_1$.
One obtains (\ref{ZZUBP3}) because of
\begin{alignat*}{1}
S_\delta(h_n) &= \omega_r(h_n, \delta)_p \leq 2^r \| h_n\|_{X^p[0, 1]} \leq 2^r
\text{ and (see (\ref{gegenAbl}))}\\
S_\delta(h_n) &= \omega_r(h_n, \delta)_p \leq \delta^r \| h_n^{(r)}\|_{X^p[0,
1]} \leq ((16n+4)\tilde{r} \pi)^r \delta^r\\ 
&\leq (\tilde{r} \pi)^r \delta^r
(20 n)^r = (\tilde{r} 20\pi)^r
\frac{\mu(\delta)}{\varphi(n)}.
\end{alignat*}
Let $g\in S_{4n}^{\tilde{r}}$, then $g$ is composed from at most $4n+1$
polynomials on $4n+1$ intervals. On each of these intervals, $g\equiv 0$ or $g$
at most has $\tilde{r}-1$ zeroes. Thus $g$ can change sign at $4n$ interval
borders and at zeroes of polynomials, and $g$ fulfills the prerequisites of
Lemma
\ref{laresonanz}
with $N=(4n+1)\cdot \tilde{r} >  4n+(4n+1)\cdot (\tilde{r}-1)$.
Due to the lemma, 
$\| h_n-g\|_{X^p[0, 1]}\geq c>0$ independent of $n$ and $g$.
Since this holds
true for all $g$, (\ref{ZUBP3}) is shown for $c_3=c$.
All preliminaries of Theorem \ref{ZUBP} are fulfilled such that counterexamples
exist as stated.
\end{proof} 

Corollary \ref{thsharpsplines} can be applied with respect to 
all activation functions $\sigma$ belonging to the class of
splines with fixed polynomial degree less than $r$ and a finite number of
knots $k$ because 
$\Phi_{n,\sigma}\subset S_{nk}^{r}$. 

Since $\Phi_{n,\sigma_h} \subset S_n^1$, Corollary \ref{thsharpsplines} directly
shows sharpness of (\ref{d1}) and (\ref{m2b}) for the Heaviside activation
function if one chooses $r=\tilde{r}=1$. 
Sharpness of (\ref{m2}) for cut 
and ReLU function follows for $r=\tilde{r}=2$ because $\Phi_{n,\sigma_c}\subset
S_{2n}^2$, $\Phi_{n,\sigma_r}\subset
S_{n}^2$.
However, the case $\omega(\delta)=\delta$ of maximum non-saturated convergence
order is excluded by condition (\ref{odd}). 
We discuss this case for
$r=\tilde{r}$. Then a simple counterexample is
$f_\omega(x):=x^r$.
For each sequence of coefficients $d_0,\dots,d_{r-1}\in\mathbb{R}$ we can apply
the fundamental theorem of algebra to find complex zeroes
$a_0,\dots,a_{r-1}\in\mathbb{C}$ such that
$$ \left|x^r - \sum_{k=0}^{r-1} d_k x^k\right| = \prod_{k=0}^{r-1} |x-a_k|. 
$$
There exists an interval $I:=(j (r+1)^{-1}, (j+1) (r+1)^{-1})\subset [0, 1]$
such that real parts of complex numbers $a_k$ are not in $I$ for all $0\leq k<r$.
Let $I_0:=((j+1/4) (r+1)^{-1}, (j+3/4) (r+1)^{-1})\subset I$. Then for all $x\in
I_0$ $$ \prod_{k=0}^{r-1} |x-a_k| \geq \left[\frac{1}{4(r+1)}\right]^r =: c_\infty >
0.
$$
This lower bound is independent of coefficients $d_k$ such that
$$ \inf \{ \|x^r - q(x)\|_{B[0, 1]} : q\in \Por\} \geq c_\infty > 0. $$
We also see that
$$ \int_0^1 \left|x^r - \sum_{k=0}^{r-1} d_k x^k\right|^p\, dx
\geq \int_{I_0} \frac{1}{[4(r+1)]^{pr}}\, dx
= \frac{\frac{1}{2(r+1)}}{[4(r+1)]^{pr}} =: c_p^p > 0,
$$
$$ \inf \{ \|x^r - q(x)\|_{L^p[0, 1]} : q\in \Por\} \geq c_p >
0. $$
Each function $g\in S_n^r$
is a polynomial of degree less than $r$ on at least $n$ intervals $(j
(2n)^{-1}, (j+1) (2n)^{-1} )$, $j\in J\subset\{0,1,\dots,2n-1\}$. For $j\in J$:
$$ \inf_{q\in \Por} \|x^r - q(x)\|_{B\left[\frac{j}{2n}, \frac{j+1}{2n}\right]}
=
   \inf_{q\in \Por} \left\|\left(\frac{x}{2n}+\frac{j}{2n}\right)^r -
   q(x)\right\|_{B[0, 1]} \geq \frac{c_\infty}{(2n)^r}.
$$
Thus, $E(S_n^r,x^r) \neq o\left(\frac{1}{n^r}\right)$. In case of $L^p$-spaces,
we similarly obtain with substitution $u=2n x-j$
\begin{alignat*}{1}
 \inf_{q\in \Por} &\int_{\frac{j}{2n}}^{\frac{j+1}{2n}} |x^r - q(x)|^p\, dx =
   \inf_{q\in \Por} \frac{1}{2n}\int_{0}^{1}
   \left|\left(\frac{u}{2n}+\frac{j}{2n}\right)^r - q(u)\right|^p\, du\\ 
   & =
   \inf_{q\in \Por} \frac{1}{2n}\int_{0}^{1}
   \frac{1}{(2n)^{pr}}\left|u^r - q(u)\right|^p\, du
   \geq
   \frac{c_p^p}{2n\cdot (2n)^{pr}}.
\end{alignat*}
Sharpness is demonstrated by combining lower estimates of all $n$ subintervals:
$$ E(S_n^r,x^r)_p^p \geq n  \frac{c_p^p}{(2n)^{pr+1}}, \quad E(S_n^r,x^r)_p \geq
\frac{c_p}{2^{r+\frac{1}{p}} n^{r}}\neq o\left(\frac{1}{n^r}\right).
$$

Although our counterexample is arbitrarily often differentiable, the
convergence order is limited to $n^{-r}$. Reason is the 
definition of the activation function by piecewise polynomials. There is no such limitation for
activation functions that are arbitrarily often differentiable 
on an interval without being a polynomial, see Theorem
\ref{directsigmoid2}.
Thus, neural networks based on smooth non-polynomial activation functions might
approximate better if smooth functions have to be learned.

Theorem 3 in
\cite{Chen93} states for the Heaviside function that for each $n\in \mathbb{N}$
a function $f_n\in C[0, 1]$ exits such that the error of best
uniform approximation exactly equals $\omega_1\left(f_n,
\frac{1}{2(n+1)}\right)$.
This is used to show optimality of the constant. Functions $f_n$
might be different for different $n$. One does not get the
condensed sharpness result of Corollary
\ref{thsharpsplines}.

Another relevant example of a spline of fixed polynomial degree with a finite
number of knots is the square
non-linearity (SQNL) activation function $\sigma(x):= \sgn(x)$ for $|x|>2$ and $\sigma(x):= x-\sgn(x)\cdot x^2/4$ for $|x|\leq 2$. Because
$\sigma$, restricted to each of the four sub-intervals of piecewise definition,
is a polynomial of degree two, we can choose $\tilde{r}=3$.

The proof of Corollary \ref{thsharpsplines} is based on Lemma \ref{laresonanz}. 
This argument can be also used to deal with
rational activation functions $\sigma(x)=q_1(x)/q_2(x)$ where 
$q_1,q_2\not\equiv 0$ are polynomials of degree at most $\rho$. Then non-zero
functions $g\in \Phi_{n,\sigma}$ do have at most $\rho n$ zeroes and $\rho n$ poles such
that there is no change of sign on $N+1$ intervals, $N=2\rho n$. Thus,
Corollary 1 can be extended to neural network approximation with rational
activation functions in a straight forward manner.

Whereas the direct estimate (\ref{m2})
for cut and ReLU functions is based on linear best approximation, 
the counterexamples hold for non-linear best approximation. 
Thus, error bounds in terms of moduli of smoothness may not be able to express
the advantages of non-linear free knot spline approximation in contrast to fixed
knot spline approximation (cf.\ \cite{Sanguineti99}). For an error measured in
an $L^p$ norm with an order like $n^{-\alpha}$, smoothness only 
is required in $L^q$, $q:=1/(\alpha+1/p)$, see (\ref{equivalence}) and
\cite[p.~368]{DeVore93}.

\begin{corollary}[Inverse Tangent]\label{thsharp3}
Let $\sigma=\sigma_a$ be the sigmoid function based on the
inverse tangent function, $r\in\mathbb{N}$, and $1\leq p\leq\infty$.
For each abstract modulus of smoothness $\omega$ satisfying 
(\ref{odd}),
there exists a counterexample
$f_{\omega}\in X^p[0,1]$ such that 
$$ \omega_r(f_{\omega}, \delta)_p = {O}\left(\omega(\delta^r)\right)\text{ and
} E(\Phi_{n,\sigma_a}, f_{\omega})_p \neq
o\left(\omega\left(\frac{1}{n^r}\right)\right).
$$
\end{corollary} 
The corollary shows sharpness of the error bound in Theorem
\ref{directsigmoid2} applied to the arbitrarily often differentiable
function $\sigma_a$.

\begin{proof}
Similarly to the proof of Corollary \ref{thsharpsplines}, we
apply Theorem \ref{ZUBP} with parameters $X=X^p[0, 1]$, $E_n(f):=
E(\Phi_{n,\sigma_a}, f)_p$, $S_\delta(f) := \omega_r(f, \delta)_p$,
$\mu(\delta):=\delta^r$, $\varphi(x)=1/x^r$, and 
$$h_n(x) := \sin\left(16n\cdot 2\pi x\right) =  \sin\left(2N \cdot 2\pi x\right)
$$
with $N=N(n):=8n$,
such that condition (\ref{ZUBP1}) is obvious and (\ref{ZZUBP3}) can be shown 
by estimating the modulus in terms of the $r$th derivative of $h_n$
with (\ref{gegenAbl}).
Let $g\in \Phi_{4n,\sigma_a}$, 
$$g(x)=\sum_{k=1}^{4n} a_k
\left(\frac{1}{2}+\frac{1}{\pi}\arctan(b_kx+c_k)\right).$$ Then 
$$ g'(x)= \sum_{k=1}^{4n}
\frac{a_k b_k}{\pi} \frac{1}{1+(b_kx+c_k)^2} = \frac{s(x)}{q(x)}
$$
where $s(x)$ is a polynomial of degree $2(4n-1)$, and $q(x)$ is a polynomial of
degree $8n$. If $g$ is not constant then $g'$ at most has $8n-2$ zeroes and $f$
at most has $8n-1$ zeroes due to the mean value theorem (Rolle's theorem).
In both cases, the requirements of Lemma \ref{laresonanz} are fulfilled with
$N(n)=8n>8n-1$ such that
$\| h_n-g\|_{X^p[0, 1]}\geq c>0$ independent of $n$ and $g$.
Since $g$ can be chosen
arbitrarily, (\ref{ZUBP3}) is shown with $E_{4n} h_n \geq c>0$.
\end{proof}

Whereas lower estimates for sums of $n$ inverse tangent functions are
easily obtained by considering $O(n)$ zeroes of their derivatives, sums of $n$
logistic functions (or hyperbolic tangent functions) might have an
exponential number of zeroes.
To illustrate the problem in the context of
Theorem \ref{ZUBP}, let 
\begin{equation}
g(x):= \sum_{k=1}^{4n}
\frac{a_k}{1+e^{-c_k}(e^{-b_k})^x} \in \Phi_{4n,\sigma_l}.\label{eqtry}
\end{equation}
Using a common denominator, the numerator is a sum of type 
$$\sum_{k=1}^{m}
\alpha_k \kappa_k^x$$ 
for some $\kappa_k>0$ and $m<2^{4n}$. According to
\cite{Tossavainen07}, such a function has at most $m-1<16^n-1$ zeroes, or
it equals the zero function. Therefore, an interval $[k (16)^{-n},
(k+1) (16)^{-n}]$ exists on which $g$ does not change its sign.
By using a resonance sequence $h_n(x) := \sin\left(16^n\cdot 2\pi x\right)$,
one gets $E(\Phi_{4n,\sigma_l}, h_n) \geq 1$. But factor $16^n$ is by far too
large. One has to choose $\phi(x):=1/16^x$ and $\mu(\delta):=\delta$ to
obtain a ``counterexample'' $f_\omega$ with 
\begin{equation}
E(\Phi_{n,\sigma_l}, f_\omega)
\in O\left(\omega\left(\frac{1}{n}\right)\right)\text{ and }
E(\Phi_{n,\sigma_l}, f_\omega) \neq o\left(\omega\left(\frac{1}{16^n}\right)\right).\label{bad}
\end{equation}
The gap between rates
is obvious. 
The same difficulties do not only occur for the logistic function but also
for other activation functions based on $\exp(x)$ like the softmax
function $\sigma_m(x):=\log(\exp(x) + 1)$.
Similar to (\ref{trick}),
$$ \frac{\partial}{\partial c} \sigma_m(bx+c) = \frac{\exp(bx+c)}{\exp(bx+c)+1}
= \sigma_l(bx+c).$$
Thus, sums of $n$ logistic functions can be approximated uniformly
and arbitrarily well by sums of differential quotients that can be written by
$2n$ softmax functions.
A lower bound for approximation with $\sigma_m$
would also imply a similar bound for $\sigma_l$ and upper bounds for
approximation with $\sigma_l$ imply upper bounds for $\sigma_m$.

With respect to the logistic function, a better estimate than
(\ref{bad}) is possible.
It can be condensed from a sequence of counterexamples that is derived in
\cite{Maiorov99}.
However, we show that the Vapnik-Chervonenkis dimension (VC
dimension) of related function spaces can also be used to prove sharpness. This
is a rather general approach since many VC dimension estimates are known.
 
Let $X$ be a set and ${\cal A}$ a family of subsets of
$X$. Throughout this paper, $X$ can be assumed to be finite.
One says that ${\cal A}$ shatters a set $S\subset X$ if and only if each
subset $B\subset S$ can be written as $B=S\cap A$ for a set $A\in {\cal A}$. 
The VC dimension of ${\cal A}$ is defined via
\begin{alignat*}{1}
\operatorname{VC-dim}({\cal A}) :=  \sup \{&k\in\mathbb{N} : \exists S\subset X
\text{ with cardinality }\\
& |S|=k \text{ such that } S \text{ is shattered
by } {\cal A}\}.
\end{alignat*}
This general definition can be adapted to (non-linear) function spaces $V$
that consist of functions $g:X\to\mathbb{R}$ on a (finite) set
$X\subset\mathbb{R}$.
By applying Heaviside-function $\sigma_h$,
let
\begin{alignat*}{1}
 {\cal A} := \{ A\subset X : \exists g\in V : & (\forall x\in A:
\sigma_h(g(x))=1)\quad \wedge\\  
& (\forall x\in X\setminus A: \sigma_h(g(x))=0) \}.
\end{alignat*}
Then the VC dimension of function space $V$ is defined as
$\operatorname{VC-dim}(V):=\operatorname{VC-dim}({\cal A})$.
This is the largest
number $m\in\mathbb{N}$ for which $m$ points $x_1$, \dots, $x_m \in X$ exist such that for each sign
sequence $s_1,\dots, s_m\in\{-1, 1\}$ a function $g\in V$
can be found that fulfills
\begin{equation}
s_i \cdot \left(\sigma_h(g(x_i))-\frac{1}{2}\right) > 0,\quad 1\leq i\leq
m,\label{defVC}
\end{equation}
cf.~\cite{Bartlett1996}.
The VC dimension is an
indicator for the number of degrees of freedom in the construction of $V$.
Condition (\ref{defVC}) is equivalent to
$$\sigma_h(g(x_i)) = \sigma_h(s_i),\quad 1\leq i\leq m.
$$

\begin{corollary}[Sharpness due to VC Dimension]\label{thVC}
Let 
$(V_n)_{n=1}^\infty$ be a sequence of (non-linear) function spaces $V_n\subset
B[0, 1]$ such that 
$$E_n(f):=\inf \{ \| f-g\|_{B[0, 1]} : g\in
V_n\}$$ 
fulfills conditions 
(\ref{semisub})--(\ref{mono}) 
on Banach space $C[0, 1]$.
Let $\tau :\mathbb{N} \to \mathbb{N}$ and 
$$G_n:=\left\{ \frac{j}{\tau(n)} : j\in\{0,1,\dots,\tau(n)\}\right\}$$ 
be an equidistant grid on the interval $[0, 1]$. We restrict functions in
$V_n$
 to this grid:
$$ V_{n,\tau(n)}  := \{ h:G_n\to\mathbb{R} : h(x)=g(x)
\text{ for a function } g\in V_n\}.$$
Let function $\varphi(x)$ be
defined as in Theorem \ref{ZUBP} such that (\ref{faktor}) holds true. If, for a
constant $C>0$, function $\tau$ fulfills
\begin{eqnarray}
\operatorname{VC-dim}(V_{n,\tau(n)}) &<& \tau(n),\label{VCcond}\\
\tau(4n) &\leq& \frac{C}{\varphi(n)}, \label{alphacond}
\end{eqnarray}
for all $n\geq n_0\in\mathbb{N}$
then for $r\in\mathbb{N}$ 
and each abstract modulus of smoothness $\omega$ satisfying 
(\ref{odd}),
a counterexample
$f_{\omega}\in C[0,1]$ exists such that 
$$ \omega_r(f_{\omega}, \delta) = {O}\left(\omega(\delta^r)\right)\text{ and
} E_n(f_{\omega}) \neq
o\left(\omega\left([\varphi(n)]^r\right)\right).
$$
\end{corollary} 

\begin{proof}
Let $n\geq n_0/4$.
Due to (\ref{VCcond}), a sign sequence $s_{0},\dots,s_{\tau(4n)}\in\{-1, 1\}$
exists such that for each function $g\in
V_{4n}$ there is a point $x_0=\frac{i}{\tau(4n)} \in G_{4n}$ such
that $\sigma_h(g(x_0)) \neq \sigma_h(s_{i})$.

We utilize this sign sequence to construct resonance elements. It
is well known, that auxiliary function
\begin{equation*}
 h(x) := \left\{ \begin{array}{ll}
\exp\left(1-\frac{1}{1-x^2}\right) & \text{for } |x|<1,\\
0 & \text{for } |x|\geq 1,
\end{array}\right.
\end{equation*}
is arbitrarily often differentiable on the real axis, $h(0)=1$,
$\|h\|_{B(\mathbb{R})}=1$.
This function becomes the building block for the resonance elements:
$$ h_n(x):= \sum_{i=0}^{\tau(4n)} s_{i} \cdot
h\left( 2\tau(4n)\left(x-\frac{i}{\tau(4n)}\right)\right).$$ 
The interior of the support of
summands is non-overlapping, i.e., $\|h_n\|_{B[0,1]}\leq 1$, and 
because of (\ref{alphacond}) norm
$\|h_n^{(r)}\|_{B[0,1]}$ of the $r$th derivative is in
$O([\varphi(n)]^{-r})$.

We apply Theorem \ref{ZUBP}
with 
$$X=C[0, 1]\text{, }
S_\delta(f) := \omega_r(f, \delta)\text{, } \mu(\delta):=\delta^r,$$
$E_n(f)$ as defined in the theorem,
and
resonance elements $h_n(x)$ that represent the existing sign sequence.
Function $[\varphi(x)]^r$ fulfills the requirements of function $\varphi(x)$ in
Theorem \ref{ZUBP}.

Then conditions (\ref{ZUBP1}) and (\ref{ZZUBP3}) are fulfilled due to the norms
of $h_n$ and its derivatives, cf.\ (\ref{gegenAbl}). Due to the initial
argument of the proof, for each $g\in V_{4n}$ at least one point $x_0=i
\tau(4n)^{-1}$, $0\leq i\leq \tau(4n)$, exists such that we observe $\sigma_h(g(x_0))\neq\sigma_h(h_n(x_0))$.
Since $|h_n(x_0)|=1$, we get $\|h_n-g\|_{B[0,1]} \geq |h_n(x_0)-g(x_0)| \geq 1$,
and (\ref{ZUBP3}) because $E_{4n} h_n \geq 1$.
\end{proof}

\begin{corollary}[Logistic Function]\label{thsharp5}
Let $\sigma=\sigma_l$ be the logistic function and $r\in\mathbb{N}$. 
For each abstract modulus of smoothness $\omega$ satisfying 
(\ref{odd}),
a counterexample
$f_{\omega}\in C[0,1]$ exists such that 
$$ \omega_r(f_{\omega}, \delta) = {O}\left(\omega(\delta^r)\right)\text{ and
} E(\Phi_{n,\sigma_l}, f_{\omega}) \neq
o\left(\omega\left(\frac{1}{(n [1+\log_2(n)])^r}\right)\right).
$$
\end{corollary} 
The corollary extends the Theorem of Maiorov and Meir for worst case
approximation with sigmoid functions in the case $p=\infty$ to Lipschitz
classes and
one condensed counterexample (instead of a sequence),
see \cite[p.~99]{Maiorov99}. The sharpness estimate also holds in $L^p[0, 1]$,
$1\leq p<\infty$. For all these spaces, one can apply Theorem \ref{ZUBP}
directly with the sequence of counterexamples constructed in \cite[Lemma 7,
p.~99]{Maiorov99}. 
Even more generally, Theorem 1 in \cite{Ratsaby99} 
utilizes pseudo-dimension, a generalization of VC
dimension, to provide bounded
sequences of counterexamples in Sobolev spaces. 

\begin{proof}
We apply Corollary \ref{thVC} in connection with a result concerning the
VC dimension of function space ($D\in\mathbb{N}$)
\begin{eqnarray*}
 \Phi_{n,\sigma_l}^D &:=&\big\{g : \{-D,-D+1,\dots, D\} \to\mathbb{R} :\\
&&\qquad g(x)=a_0+ \sum_{k=1}^{n} a_k \sigma_l(b_k x+c_k) :
a_0,a_k,b_k,c_k\in\mathbb{R}\big\},
\end{eqnarray*}
see paper \cite{Bartlett1996} that is based on \cite{Goldberg1995}, cf.~\cite{KARPINSKI1997169}.
Functions are defined on a discrete set with $2D+1$ elements, and in contrast to
the definition of $\Phi_{n,\sigma}$, a constant function with coefficient
$a_0$ is added.

According to Theorem 2 in \cite{Bartlett1996}, the VC dimension of
$\Phi_{n,\sigma_l}^D$ is upper bounded by (n large enough)
$$2 \cdot (3\cdot n+1) \cdot
\log_2 (24e(3\cdot n+1)D),$$
i.e., there exists $n_0\in\mathbb{N}$, $n_0\geq 2$, and a constant $C>0$
such that for all $n\geq n_0$ 
$$\operatorname{VC-dim}(\Phi_{n,\sigma_l}^D) \leq C n[\log_2(n) +
\log_2(D)].
$$
Since $\lim_{E\to\infty} \frac{1+\log_2(E)}{E}=0$, we can choose a
constant $E>1$ such that 
$$\frac{1+\log_2(E)}{E} < \frac{1}{4C}.$$
With this constant, we define $D=D(n):= \lfloor E n(1+\log_2(n))\rfloor$
such that the VC dimension of
$\Phi_{n,\sigma_l}^D$ is less than $D$ for $n\geq n_0$:
\begin{eqnarray*}
\operatorname{VC-dim}(\Phi_{n,\sigma_l}^D) &\leq&
C n[\log_2(n) + \log_2(E n(1+\log_2(n)) )]\\
&\leq&
C n[\log_2(n) + \log_2(E n(\log_2(n)+\log_2(n)) )]\\
&\leq&
C n[2\log_2(n) + \log_2(E) +\log_2(2\log_2(n)) )]\\
&\leq&
C n[3\log_2(n) + \log_2(E) + 1  )]
\leq 4C n\log_2(n)[1+\log_2(E)]\\
&<& E n\log_2(n) \leq \lfloor  E n(1+\log_2(n)) \rfloor
= D.
\end{eqnarray*}

By applying an affine transform that maps interval $[-D, D]$ to $[0,
1]$ and by omitting constant function $a_0$, we immediately see that for
$V_n:=\Phi_{n,\sigma_l}$ and $\tau(n):=2D(n)$
$$
\operatorname{VC-dim}(V_{n,\tau(n)}) < D(n) =\frac{\tau(n)}{2} < \tau(n)
$$
such that (\ref{VCcond}) is fulfilled.

We define strictly decreasing 
$$\varphi(x):=\frac{1}{x[1+\log_2(x)]}.$$
Obviously, $\lim_{x\to\infty} \varphi(x) = 0$. Condition (\ref{faktor}) holds:
Let $x>X_0(\lambda):=\lambda^{-2}$. Then
$\log_2(\lambda) > -\log_2(x)/2$ and
$$
\varphi(\lambda x) =\frac{1}{\lambda} \frac{1}{x(1+\log_2(x)+\log_2(\lambda))}
< \frac{1}{\lambda} \frac{1}{x(1+\frac{1}{2}\log_2(x))}
< \frac{2}{\lambda} \varphi(x).
$$

Also, (\ref{alphacond}) is fulfilled:  
\begin{alignat*}{1}
\tau(4n) &=2D(4n) =  2E\cdot 4n(1+\log_2(4n)) 
= 8E\cdot n(3+\log_2(n))\\ 
&< 24E\cdot n(1+\log_2(n)) =\frac{24E}{\varphi(n)}.
\end{alignat*}
Thus, all prerequisites of Corollary \ref{thVC} are shown.
\end{proof}

The corollary improves (\ref{bad}): There exists a
counterexample $f_{\omega}\in C[0,1]$ such that (see (\ref{jacksonneu}),
(\ref{jacksonneu2}))
\begin{alignat}{1}
\omega_r(f_{\omega}, \delta) &\in {O}\left(\omega(\delta^r)\right),\nonumber\\
E(\Phi_{n,\sigma_l}, f_\omega)
&\in O\left(\omega\left(\frac{1}{n^r}\right)\right)\text{ and }
E(\Phi_{n,\sigma_l}, f_\omega)
\neq
o\left(\omega\left(\frac{1}{n^r[1+\log_2(n)]^r}\right)\right).\label{bad2}
\end{alignat}

The proof is based on the $O(n\log_2(n))$ estimate of VC dimension in 
\cite{Bartlett1996}. This requires functions to be defined on a finite grid.
Without this prerequisite, the VC dimension is in $\Omega(n^2)$, see
\cite[p.~235]{Shai}. The referenced book also deals with the case that 
all weights are restricted to
floating point numbers with a fixed number of bits. Then the VC dimension
becomes bounded by $O(n)$ without the need for the $\log$-factor. However, 
direct upper error bounds (\ref{jacksonneu})
and (\ref{jacksonneu2}) are proved for real-valued
weights only.

The preceding corollary is a prototype for proving sharpness based on known
VC dimensions. Also at the price of a log-factor, the VC dimension
estimate for radial basis functions in \cite{Bartlett1996} or \cite{Schmitt01} can be used similarly in
connection with Corollary \ref{thVC} to construct counterexamples.
The sharpness results for
Heaviside, cut, ReLU and inverse tangent activation functions shown above for
$p=\infty$ can also be obtained with Corollary \ref{thVC} by proving that VC
dimensions of corresponding function spaces $\Phi_{n,\sigma}$ are in $O(n)$
(whereas the result of \cite{Bartlett1998} only provides an $O(n\log(n))$
bound).
This in turn can be shown by estimating the maximum number of zeroes 
like in the proof of the next corollary and
in
the same manner as in the proofs of Corollaries \ref{thsharpsplines} and \ref{thsharp3}.

The problem
of different rates in upper and lower bounds
arises because different scaling coefficients $b_k$ are allowed. 
In the case of uniform scaling, i.e.\ all coefficients $b_k$ in (\ref{eqtry}) 
have the same value $b_k=B=B(n)$,
the number of zeroes is bounded by $4n-1$ instead of $16^n-2$.
Let
$$ \tilde{\Phi}_{n,\sigma_l}:=\left\{g : [0,1]\to\mathbb{R} : g(x)=\sum_{k=1}^n
a_k \sigma_l(B x+c_k) :
a_k,B,c_k\in\mathbb{R}\right\}$$
be the non-linear function space generated by uniform scaling.
Because the 
quasi-inter\-po\-la\-tion operators used in the proof of Debao's direct estimate
(\ref{d1}), see \cite{Chen93}, are defined using such uniform scaling, see
\cite[p.~172]{Cheney2000}, the error bound $$ E(\tilde{\Phi}_{n,\sigma_l}, f) 
:= \inf\{\|f-g\|_{B[0, 1]} : g\in \tilde{\Phi}_{n,\sigma_l}\}
\leq \omega_1\left(f, \frac{1}{n}\right)
$$
holds. This bound is sharp:

\begin{corollary}[Logistic Function with Restriction]\label{thsharp4}
Let $\sigma=\sigma_l$ be the logistic function and $r\in\mathbb{N}$. 
For each abstract modulus of smoothness $\omega$ satisfying 
(\ref{odd}),
there exists a counterexample
$f_{\omega}\in C[0,1]$ such that 
$$ \omega_r(f_{\omega}, \delta) = {O}\left(\omega(\delta^r)\right)\text{ and
} E(\tilde{\Phi}_{n,\sigma_l}, f_{\omega}) \neq
o\left(\omega\left(\frac{1}{n^r}\right)\right).
$$
\end{corollary} 

To prove the corollary, we apply following lemma.

\begin{lemma}\label{lemmazeores}
$V_n\subset
C[0, 1]$, $2\leq \tau(n)\in\mathbb{N}$, $G_n:=\left\{ \frac{j}{\tau(n)} :
j\in\{0,1,\dots,\tau(n)\}\right\}$, and
$V_{n,\tau(n)}  := \{ h:G_n\to\mathbb{R} : h(x)=g(x)
\text{ for a function } g\in V_n\}$ are given as in Corollary \ref{thVC}.
If $\operatorname{VC-dim}(V_{n,\tau(n)}) \geq \tau(n)$ then 
there exists a function $g\in V_n$, $g\not\equiv 0$, with a set of at least
$\lfloor \tau(n)/2\rfloor$ zero points in $[0, 1]$ such that $g$ has
non-zero function values between each two consecutive points of
this set.
\end{lemma}
\begin{proof}
Because of the VC dimension, a subset
$J\subset\{0,1,\dots,\tau(n)\}$ with $\tau(n)$ elements $j_1 < j_2 <
\dots < j_{\tau(n)}$ and a function $g\in V_n$ exist such that
$$\sigma_h\left(g\left(\frac{j_k}{\tau(n)}\right)\right) =
\frac{1+(-1)^{k+1}}{2},\, 1\leq k\leq \tau(n).$$ Then we find a zero on each
interval $[j_{2k-1} \tau(n)^{-1}, j_{2k} \tau(n)^{-1})$,
$1\leq k\leq \lfloor \tau(n)/2\rfloor$:
If $g(j_{2k-1} \tau(n)^{-1})\neq 0$ then continuous $g$ has a
zero on the interval 
$$(j_{2k-1} \tau(n)^{-1}, j_{2k} \tau(n)^{-1}).$$ 
Thus, $g$ has $\lfloor \tau(n)/2\rfloor$ zeroes on different
sub-intervals.
Between zeroes are non-zero, negative function values
$g(j_k \tau(n)^{-1})$ for even
$k$ because $\sigma_h(g(j_k \tau(n)^{-1})) = 0$.
\end{proof}

\begin{proof} (of Corollary \ref{thsharp4})
We apply Corollary \ref{thVC} with $V_n=\tilde{\Phi}_{n,\sigma_l}$ 
and $E_n(f) = E(\tilde{\Phi}_{n,\sigma_l}, f)$ such that conditions
(\ref{semisub})--(\ref{mono}) are fulfilled. Let $\tau(n):=2n$ and
$\varphi(n)=1/n$ such that (\ref{alphacond}) holds true: $\tau(4n) =8n
=8/\varphi(n)$.
Assume that $\operatorname{VC-dim}(V_{n,\tau(n)}) \geq
\tau(n)$, then according to Lemma \ref{lemmazeores} there exists a function
$f\in \tilde{\Phi}_{n,\sigma_l}$ such that $f\not\equiv 0$ has $\lfloor
\tau(n)/2\rfloor = n$ zeroes.
However, we can write $f$ as
$$f(x)= \sum_{k=1}^{n}
\frac{a_k}{1+e^{-c_k}(e^{-B})^x}=\frac{s(x)}{q(x)}.$$ 
Using a common denominator
$q(x)$, the numerator is a sum of type 
$$s(x)=\sum_{k=0}^{n-1} \alpha_k
(e^{-kB})^x$$ 
which has at most $n-1$ zeroes, see \cite{Tossavainen07}.
Because of this contradiction to $n$ zeroes, (\ref{VCcond}) is fulfilled.
%
\end{proof}
By applying Lemma \ref{laresonanz} for $N(n)=n-1$ in connection with Theorem
\ref{ZUBP}, one can also show Corollary \ref{thsharp4} for $L^p$-spaces, $1\leq
p<\infty$.

Linear VC dimension bounds were proved in \cite{Schmitt02} for radial basis
function networks with uniform width (scaling) or uniform centers. Such bounds can be
used with Corollary \ref{thVC} to prove results that are similar to Corollary
\ref{thsharp4}. Also, such a corollary 
can be shown 
for the ELU function $\sigma_e$.
However, without the restriction $b_k=B(n)$, piecewise
superposition of exponential functions leads to $O(n^2)$ zeroes of sums of ELU functions. 
Then in combination with direct estimates
Theorems \ref{directsigmoid} and \ref{directsigmoid2}, i.e.,
$E(\Phi_{n,\sigma_e}, f_{\omega})  \leq C_r \omega_r\left(f_{\omega}, \frac{1}{n}\right)$, 
we directly obtain following (improvable) result in a straightforward manner. 

\begin{corollary}[Coarse estimate for ELU activation]\label{corollaryELU}
Let $\sigma=\sigma_e$ be the ELU function and $r\in\mathbb{N}$, $n\geq\max\{2,
r\}$ (see Theorem \ref{directsigmoid}). For each abstract modulus of smoothness
$\omega$ satisfying (\ref{odd}),
there exists a counterexample
$f_\omega\in C[0,1]$ that fulfills
\begin{alignat}{1}
 E(\Phi_{n,\sigma_e}, f_{\omega})  &\leq C_r
\omega_r\left(f_{\omega},
\frac{1}{n}\right) \in {O}\left(\omega\left(\frac{1}{n^{r}}\right)\right)\text{
and }\nonumber\\  
E(\Phi_{n,\sigma_e}, f_{\omega}) &\neq
o\left(\omega\left(\frac{1}{n^{2r}}\right)\right)\label{notsharp}.
\end{alignat}
\end{corollary}
\begin{proof}
To prove the existence of a function $f_\omega$ with $\omega_r\left(f_{\omega},
\delta\right) \in {O}\left(\omega\left(\delta^{r}\right)\right)$ and
(\ref{notsharp}), we apply Corollary \ref{thVC} with $V_n=\Phi_{n,\sigma_e}$ 
and $E_n(f) = E(\Phi_{n,\sigma_e}, f)$ such that conditions
(\ref{semisub})--(\ref{mono}) are fulfilled. 
For each function $g\in V_n$ 
the interval $[0, 1]$ can be divided into at most $n+1$ subintervals such that 
on the $l$th interval $g$ equals a function $g_l$ of type
$$ g_l(x) = \gamma_l + \delta_l x + \sum_{k=1}^n \alpha_{l,k} \exp(\beta_{l,k}
x). $$
Derivative
$$ g_l'(x) = \delta_l \exp(0\cdot x) + \sum_{k=1}^n \alpha_{l,k}\beta_{l,k}
\exp(\beta_{l,k} x) $$
has at most $n$ zeroes or equals the zero function according to
\cite{Tossavainen07}. Thus, due to the mean value theorem (or Rolle's theorem),
$g_l$ has at most $n+1$ zeroes or is the zero function. By concatenating
functions $g_l$ to $g$, one observes that $g$ has at most $(n+1)^2$ different
zeroes such that $g$ does not vanish between such consecutive zero points.

Let $\tau(n):=8 n^2$ and
$\varphi(n)=1/n^2$ such that (\ref{alphacond}) holds true: $\tau(4n) = 128 n^2
=128/\varphi(n)$.
If $\operatorname{VC-dim}(V_{n,\tau(n)}) \geq
\tau(n)$ then due to Lemma \ref{lemmazeores} and because $n\geq 2$ there exists
a function in $\Phi_{n,\sigma_e}$ with at least $\lfloor
\tau(n)/2\rfloor=(2n)^2>(n+1)^2$ zeroes such that between consecutive zeroes,
the function is not the zero function.
This contradicts the previously determined number of zeroes and (\ref{VCcond})
is fulfilled.
\end{proof}

Sums of $n$ softsign functions $\varphi(x)=x/(1+|x|)$ can be expressed piecewise
by $n+1$ rational functions that each have at most $n$ zeroes. Thus, one also
has to deal with $O(n^2)$ zeroes.

\section{Conclusions}\label{secFuture}
Corollaries
\ref{thsharpsplines} and \ref{thsharp3} can be seen
in the context of Lipschitz classes. Let $r:=1$ for the Heaviside, $r:=2$ for cut and ReLU functions
and $r\in\mathbb{N}$ for the inverse tangent activation function
$\sigma=\sigma_a$.
By choosing 
$\omega(\delta):=\delta^\alpha$, a counterexample
$f_\alpha\in X^p[0, 1]$ exists for each $\alpha\in (0,r)$
such that
$$ \omega_r(f_\alpha, \delta)_p \in O\left( \delta^\alpha \right),\text{ } 
E(\Phi_{n,\sigma}, f_{\alpha})_p \in O\left(\frac{1}{n^\alpha}\right) \text {,
and }  E(\Phi_{n,\sigma}, f_{\alpha})_p \neq o\left(\frac{1}{n^\alpha}\right).
$$
With respect to Corollary \ref{thsharp5} for the logistic function, in general
no higher convergence order $\frac{1}{n^\beta}$, $\beta>\alpha$ can be expected for
functions in the Lipschitz class that is defined via 
$\operatorname{Lip}^r(\alpha, C[0, 1]):=\{f \in C[0,1] :
\omega_r(f, \delta) = O\left(\delta^\alpha\right)\}$.

In terms of (non-linear) Kolmogogrov $n$-width,
let $X:=\operatorname{Lip}^r(\alpha, C[0, 1])$. Then, for example, condensed
counterexamples $f_\alpha$ for piecewise linear
or inverse tangent activation functions and $p=\infty$ imply
\begin{alignat*}{1}
W_n &:= \inf_{b_1,\dots,b_n, c_1\dots,c_n} \sup_{f\in
\operatorname{Lip}^r(\alpha, C[0, 1])} \inf_{a_1,\dots,a_n}\left\| f(\cdot)- \sum_{k=1}^n a_k\sigma(b_k\, \cdot\, +
c_k)\right\|_{B[0,1]} \\
&\geq \inf_{b_1,\dots,b_n, c_1\dots,c_n} \inf_{a_1,\dots,a_n}\left\|
f_\alpha(\cdot)- \sum_{k=1}^n a_k\sigma(b_k\, \cdot\, + c_k)\right\|_{B[0,1]} \\
&= E(\Phi_{n,\sigma}, f_{\alpha}) \neq o\left(\frac{1}{n^\alpha}\right).
\end{alignat*}

The restriction to the univariate case of a single input node was chosen 
because of compatibility with most cited error bounds. 
However, the error of multivariate approximation with certain activation functions 
can be bounded by the error of best multivariate
polynomial approximation, see proof of Theorem 6.8 in
\cite[p.~176]{pinkus_1999}.
Thus, one can obtain estimates in terms of multivariate radial moduli of
smoothness similar to Theorem \ref{directsigmoid2} via \cite[Corollary 4, p.~139]{JohnenScherer}.
Also, Theorem \ref{ZUBP} can be
applied in a multivariate context in connection with VC dimension bounds. First
results are shown in report \cite{Goebbels20b}.

Without additional restrictions, a lower estimate for approximation with
logistic function $\sigma_l$ could only be obtained with a log-factor in
(\ref{bad2}). Thus, either direct bounds  (\ref{d1}) and (\ref{jacksonneu}) or
sharpness result (\ref{bad2}) can be improved slightly. 

\subsection*{Acknowledgment}
I would like to thank an anonymous reviewer, Michael Gref, Christian Neumann,
Peer Ueberholz, and Lorens Imhof for their valuable comments.

\bibliographystyle{spmpsci}
\bibliography{neural}

\end{document}